\RequirePackage[table]{xcolor}
\documentclass[bj,authoryear, noshowframe]{imsart} 
\makeatletter
    \def\journal@name{}
    \def\journal@id{}
\makeatother

\RequirePackage{amsthm,amsmath,amsfonts,amssymb}
\usepackage{mathtools}
\usepackage{bbm}
\usepackage{csquotes}
\usepackage{pgfplots}
\usepackage{pgfplotsthemetol}
\pgfplotsset{compat=1.11}
\usepgfplotslibrary{fillbetween}
\usepgfplotslibrary{groupplots}
\usepackage{tikz}
\usepackage{tikz-cd}
\usetikzlibrary{arrows, positioning}
\usepackage{subcaption}
\usepackage{enumitem}
\usepackage{bm}
\usepackage{algorithm2e}
\usepackage{multirow}
\usepackage{booktabs}
\usepackage{pdfpages}

\startlocaldefs
\theoremstyle{plain}
\newtheorem{theorem}{Theorem}[section]
\newtheorem{lemma}[theorem]{Lemma}
\newtheorem{proposition}[theorem]{Proposition}
\newtheorem{corollary}[theorem]{Corollary}
\theoremstyle{remark} 
\newtheorem{definition}[theorem]{Definition}
\newtheorem{remark}[theorem]{Remark}
\newtheorem{example}[theorem]{Example}
\newcommand{\R}{\ensuremath{\mathbb{R}}}
\newcommand{\C}{\ensuremath{\mathcal{C}}}	

\newcommand{\boldm}[1]{\bm{#1}}						
\newcommand{\rbraces}[1]{\left( #1 \right)} 		
\newcommand{\cbraces}[1]{\left[ #1 \right]}			
\newcommand{\curly}[1]{\left\{ #1 \right\}} 		
\newcommand{\abs}[1]{\left\lvert #1 \right\rvert} 	
\newcommand{\norm}[1]{\left\lVert #1 \right\rVert} 	

\newcommand{\indFunc}[1]{\mathbbm{1}_{#1}} 
\newcommand{\cInt}[4]{\int_{#2}^{#3} #1 \ \mathrm{d}#4}		

\newcommand{\ir}[1]{#1^\uparrow}		
\newcommand{\dr}[1]{#1^\downarrow}		

\newcommand{\CBnn}[1]{{C}^{\#}_{N_1,N_2}(#1)}	
\newcommand{\RD}{\ensuremath{R}}		
\newcommand{\RDLp}[1]{\RD_{#1}}

\newcommand{\RDrho}{\RDLp{\rho}}
\newcommand{\RDtau}{\RDLp{\tau}}
\newcommand{\RDk}{\RD_{\kappa}}
\newcommand{\RDm}{\RD_{\mu}}

\newcommand{\RDmRV}[2]{\RDm (#1 , #2)}

\newcommand{\refsecA}{S.I}
\newcommand{\refdefcopula}{S.I.1}

\newcommand{\refdefconcordance}{S.I.5}
\newcommand{\refpropertiesdecreasingrearrangement}{S.I.6}
\newcommand{\refdefmajorization}{S.I.7}
\newcommand{\refproporderproperties}{S.I.8}

\newcommand{\customAppendixHeading}[1]{\noindent\textbf{#1}\quad}

\newcommand{\HD}{}

\endlocaldefs

\begin{document}

\begin{frontmatter}
\title{Rearranged dependence measures}
\runtitle{Rearranged dependence measures}

\begin{aug}
\author[A]{\fnms{Christopher} \snm{Strothmann}\ead[label=e1]{christopher.strothmann@mathematik.tu-dortmund.de}}
\author[B]{\fnms{Holger} \snm{Dette}\ead[label=e2]{holger.dette@rub.de}}
\author[A]{\fnms{Karl Friedrich} \snm{Siburg}\ead[label=e3]{karl.f.siburg@mathematik.tu-dortmund.de}}

\address[A]{Department of Mathematics, TU Dortmund University, Vogelpothsweg 87, 44221 Dortmund, Germany, \printead{e1,e3}}
\address[B]{Department of Mathematics, Ruhr-University Bochum, Universit\"atsstra\ss e 150, 44780 Bochum, Germany, \printead{e2}}
\end{aug}

\begin{abstract}
Most of the popular dependence measures for two random variables $X$ and $Y$ (such as Pearson’s and Spearman’s correlation, Kendall's $\tau$ and Gini's  $\gamma$) vanish whenever $X$ and $Y$ are independent.
However, neither does a vanishing dependence measure necessarily imply independence, nor does a measure equal to 1 imply that one variable is a measurable function of the other.
Yet, both properties are natural properties for a convincing dependence measure.

In this paper, we present a general approach to transforming a given dependence measure into a new one which exactly characterizes independence as well as functional dependence.
Our approach uses the concept of monotone rearrangements as introduced by Hardy and Littlewood and is applicable to a broad class of measures.  
In particular, we are able to define  a rearranged Spearman's $\rho$ and a rearranged Kendall's $\tau$ which do attain the value $0$ if and only if both variables are independent, and the value $1$ if and only if one variable is a measurable function of the other.
We also present simple estimators for the rearranged dependence measures, prove their consistency and illustrate their finite sample properties by means of a simulation study and a data example.
\end{abstract}

\begin{keyword}[class=MSC]
\kwd[Primary ]{62H20}
\kwd[; secondary ]{62H05} 
\end{keyword}

\begin{keyword}
\kwd{measure of dependence}
\kwd{coefficient of correlation}
\kwd{decreasing rearrangement}
\kwd{copula}
\end{keyword}

\end{frontmatter}

\section{Introduction} 
\def\theequation{1.\arabic{equation}}
\setcounter{equation}{0}

One of the most fundamental problems in statistics is to measure the association between two random variables  $X$ and $Y$ based on a sample of independent identically  distributed observations $(X_1, Y_1) , \ldots , (X_n, Y_n)$, and numerous proposals have been made for this purpose.
These measures  usually vary in the interval $[0,1]$ or  $[-1,1]$, and vanish if the variables are independent. 
Moreover, many of these measures, including the frequently used Pearson’s and Spearman’s correlation, Kendall's $\tau$ and Gini's  $\gamma$, are very powerful to detect linear and monotone dependencies.
On the other hand, in general, a vanishing dependence measure (such as  Pearson's coefficient) only implies  independence of $X$ and $Y$ under quite restrictive additional assumptions (such as a normal distribution), and it is a well known fact  that many of these measures cannot detect non-monotone associations.

Several authors have proposed solutions to this problem by introducing alternative dependence measures, but mainly in the context of testing for independence. 
Among the many contributions, we mention exemplary the early work of \cite{blum1961,rosenblatt1975,Schweizer.1981,CSORGO1985} and the more recent papers by  \cite{szekely2007,gretton2008,bergsma2014} and \cite{zhang2019}. 
However, as pointed out by  \cite{Chatterjee.2020}, these measures are designed primarily for testing independence, and not for measuring the strength of the relationship between the variables.  
In the same  paper, a correlation coefficient is presented which estimates a (population) measure $\mu$ of the  dependence between two random variables $X$ and $Y$ with the following properties:
\begin{enumerate}
\smallskip 
\item[(1.1)] ~~~~~~ $ 0\leq \mu(X,Y)\leq 1$ \smallskip 
\item[(1.2)]~~~~~~ $\mu(X,Y) = 0$ if, and only if, $X$ and $Y$ are independent \smallskip 
\item[(1.3)] ~~~~~~  $\mu(X,Y) = 1$ if, and only if, $Y=f(X)$ for some measurable function $f: \mathbb{R} \rightarrow \mathbb{R}$.
\smallskip 
\end{enumerate}
For continuous distributions the measure $\mu$ has been introduced and studied in \cite{Dette.2012} who also proposed a kernel based estimator for it.
Since its introduction, Chatterjee's correlation coefficient has found considerable attention in the literature  \citep[see][among others]{caobickel2020,debghosalsen2020,gamboaetal2020,shidrtonhan2021a,shidrtonhan2021b,auddy2021exact,linhan2021}, which underlines the demand for dependence measures possessing the above properties (1.1)--(1.3).

This paper takes a quite different viewpoint on this problem by formulating the following question:

\smallskip 

\noindent\emph{Is it possible to transform a given dependence measure in such a way that the new dependence measure satisfies properties (1.1)--(1.3)?} 

\smallskip

Our answer to this question  is affirmative. More precisely, we will show that there exists a well defined  transformation $\mu \mapsto R_\mu$ with the following property. Whenever the  dependence measure $\mu$ satisfies  the axioms (1.1) to (1.3) on the set of \emph{stochastically increasing} continuous distributions, the new dependence measure $R_\mu$ will satisfy (1.1) to (1.3) on the set of \emph{all} continuous distributions.  
By definition, a pair $(X,Y)$ of random variables is stochastically increasing if the function $x\mapsto \mathbb{P}(Y\leq y \mid X=x)$ is decreasing for each fixed $y$ \citep[see, e.g.][]{Nelsen.2006}. 
This property was also discussed earlier in \cite{Lehmann.1959} under the term {\it  positive regression dependence}. 
 
The transformed  dependence measure  $R_\mu$ will be called the {\it rearranged dependence measure}. 
It turns out that the new transformation is applicable to many of the classical dependence measures and, consequently, enables us to define rearranged dependence measures  such as the rearranged  Spearman's $\rho$ and the rearranged Kendall's $\tau$, all of which satisfy properties (1.1)--(1.3).

Our approach is based on  a classical concept from majorization theory which is called {\it monotone rearrangement} \cite[see, for instance,][]{hardylittlewoodpolya,Ryff.1965,Ryff.1970}.
In the last decades, monotone rearrangements have found considerable interest in the statistical literature. 
For example, \cite{detneupil2006,chern2009,dragi2019,camirand2022} used this concept to define  (smooth)  monotone estimates, while \cite{detvol2008,chern2010} successfully applied  rearrangements  techniques  to define quantile regression estimates  without crossing.
Recently, \cite{detwu2019} used monotone rearrangements  to detect relevant changes in a (not necessarily monotone) trend of a non-stationary time series.

Our paper is organized as follows. 
In Section~\ref{sec2}, we recall the  concept of monotone rearrangements and introduce our transformation of  a given dependence measure to a new measure with the desired properties (1.1)--(1.3) in several steps. 
First, we characterize the dependence measure  $\mu (X,Y) = \mu (C)$ in terms of the copula $C$ of the corresponding distribution function of $(X,Y)$.
Then we apply a monotone rearrangement to the partial derivative of $C$ with respect to its first argument, which essentially constitutes the conditional distribution\footnote{$F_X$ and $F_Y$ denote the marginal distributions of $X$ and $Y$, respectively.}  $u \mapsto \mathbb{P}(F_Y(Y) \leq v \mid  F_X(X)=u)$, and integrate it with respect to the conditioning coordinate. 
The resulting rearranged copula is denoted by $\ir{C}$ and, roughly speaking, it can be shown that the {\it rearranged dependence measure} $$ R_\mu (C) := \mu (\ir{C}) $$ satisfies the desired properties (1.1)--(1.3).
In Section~\ref{sec3}, we propose an estimate of the  rearranged dependence measure $R_\mu (C) $,   which is obtained by applying the procedure   to the so-called checkerboard copula \citep[see][for example]{Li.1997}. 
We also prove consistency of the estimate and illustrate the finite sample properties of our approach by means of a small simulation study in Section~\ref{sec4}. Finally, all proofs are deferred to appendices and the online supplement which also contains some general results on monotone rearrangements, used for our theoretical arguments.

\section{Dependence measures with properties (1.1)--(1.3)} \label{sec2}
 \def\theequation{2.\arabic{equation}}
\setcounter{equation}{0}

In this section, we  construct  a rearranging transformation which 
assigns to  some given dependence measure $\mu$ a new measure $R_{\mu}$ with the desired properties (1.1)--(1.3).
We also discuss some further useful properties of the rearranged  measure.
To be precise, let $(X,Y)$ denote a $2$-dimensional  random vector
with  continuous distribution function $F$ and marginal distribution functions  $F_X$ and $F_Y$. 
The dependence  structure of $X$ and $Y$ is completely encoded in the (unique) copula $C=C_{X,Y}$ (see Definition~\refdefcopula{} in the Supplementary Material (\cite{Strothmann_supplement.2022})) 
defined by the equation $$ C(F_X(x) ,F_Y(y) ) = F(x,y) $$
as described, for instance, in \cite{Nelsen.2006}. The class of all copulas corresponding to continuous $2$-dimensional distributions is denoted by  $\C$.

The proofs of all the results in this section are deferred to Appendix~\ref{appendix:section2} and the supplementary material.

\subsection{New dependence measures by monotone rearrangements} 
\label{sec21}

We always consider dependence measures of $(X,Y)$ as functions of the copula $C=C_{X,Y}$ and consequently use the notations $\mu (X,Y)$ and $\mu (C) $ interchangeably.
The key ingredient is a rearrangement of the conditional distribution functions 
\begin{equation} \label{det1}
u \mapsto \mathbb{P}(F_Y(Y) \leq v \mid F_X(X)= u) = \partial_1 C(u,v) 
:= {\frac{\partial}{\partial u}}C(u,v) 
\end{equation}
of the vector $(F_X(X), F_Y(Y)) $.
Note that the partial derivative $ \partial_1 C(u, v)$ is only defined almost everywhere.
We will suppress this fact in our notation for the remainder of this article.

\begin{definition} \label{def:sicopula}
A copula $C \in \C $ is called \emph{stochastically increasing (resp.\ decreasing)} if $u\mapsto \partial_1 C(u,v)$ is decreasing (resp.\ increasing) for each $v$.
The class of all  stochastically increasing copulas  is denoted by $\ir{\C}$. 
A copula $C$ is called \emph{stochastically monotone} if it is either stochastically increasing or decreasing.
Similarly, a random variable $Y$ is stochastically increasing (resp.\ decreasing/monotone) in $X$ if $C_{XY}$ is stochastically increasing (resp.\ decreasing/monotone).   
\end{definition}

We  will now introduce a procedure transforming an arbitrary copula into a stochastically increasing one. 
It is based on the  monotone rearrangement of a  univariate function,  which is a classical  concept  in majorization theory \citep[see, for example,][]{Chong.1971, Bennett.1988}. 
Namely, if $\lambda$ denotes the Lebesgue measure and $f: [0,1]\to \R$ is a Borel measurable function, then the {\it decreasing rearrangement}  $f^*:  [0,1] \to \mathbb{R} $
 of $f$  is defined by 
 \begin{equation}
  \label{hd11}   
f^*(t) := \inf \{x \mid \lambda \rbraces{\{u \in [0, 1] \mid f(u) > x \}} \leq t \} ~.
 \end{equation}
 Obviously, the function $f^*$ 
 is a decreasing function and we have $f^* =f$ whenever $f$ is decreasing and right-continuous.

\begin{definition}  \label{def2.2}
{ 
The {\it stochastically increasing rearrangement}, (SI)-rearrangement in short, of a copula $C \in \mathcal{C} $ is defined as
\begin{align} \label{det102}
\ir{C}(u, v) := \cInt{(\partial_1 C)^*(s, v)}{0}{u}{s}  
\end{align}
where the rearrangement  \eqref{hd11} is applied to the first coordinate of $\partial_1 C(u,v)$.
}
\end{definition}

Our next result shows that $\ir{C}$ defines in fact a copula.\footnote{The analogous definition of the stochastically decreasing rearrangement copula $\dr{C}$ is given and discussed in Appendix~\ref{appendix:section2}; see also \citep{Ansari2021}. All subsequent theoretical results can be stated and proven for either $\ir{C}$ or $\dr{C}$.}

\begin{theorem} \label{thm:rearrangement}
The (SI)-rearrangement $\ir{C}$ of a copula $C$ is a stochastically increasing copula. Moreover,  $\ir{C} =C $ if and only if $C$ is stochastically increasing itself. 
\end{theorem}

For a given dependence measure $\mu$, we now define 
a new dependence  measure by 
\begin{equation}
    \label{hd1}
R_{\mu}(C) := \mu(\ir{C}) .
\end{equation} 
 We call $R_{\mu}$ the \emph{rearranged dependence measure} obtained from $\mu$. Note that, in general,  $R_{\mu}$ 
 differs from $\mu$ and hence yields a new measure of dependence.
 Our main result is the following.

\begin{theorem} \label{thm:good_measure_R}
Suppose $\mu$ is a dependence measure which, when restricted to the set 
$\ir{\C}$, satisfies the properties (1.1)--(1.3). Then 
the rearranged dependence measure $R_\mu $  satisfies  the properties (1.1)--(1.3) 
 on the whole set  $\C$.
\end{theorem}

\begin{remark} 
\label{rem1} 
{\rm Recently, dependence measures with the properties (1.1)--(1.3) have found considerable attention in the literature. For example, \cite{Trutschnig.2011} defined the measure 
 $$
 \zeta_1(C)  = 3 \cInt{\cInt{ \abs{\partial_1 C(u, v) - v}}{0}{1}{u}}{0}{1}{v} , 
 $$
while  \cite{Dette.2012} and \cite{Chatterjee.2020} considered (and proposed estimates for) the measure
\begin{equation}
    \label{detb}
r(C) = 6 \cInt{\cInt{ \rbraces{\partial_1 C(u, v) - v}^2}{0}{1}{u}}{0}{1}{v} .
\end{equation}
It will be shown in Appendix~\ref{appendix:section2} that the stochastically increasing rearrangement captures the entire information about the degree of dependence as defined by these measures in the sense that
\begin{equation} \label{det2}
\zeta_1(C) = \zeta_1(\ir{C}) \text{ as well as } r(C) = r(\ir{C}).
\end{equation}}
\end{remark}

\subsection{Examples} \label{sec22} 

In  this section, we  illustrate the rearrangement approach by a couple of examples. 
In particular, our method  is applicable to construct a rearranged Spearman's $\rho$ or Kendall's $\tau$ from those classical measures of concordance. 
Moreover,  we derive some interesting properties of the rearranged dependence measures.

\begin{example}[Schweizer-Wolff measures] \label{prop:measure_lp}
{\rm 
Let $\Pi (u,v)  = u v$ denote the independence copula. Each $L^p$-norm with $1 \leq p < \infty$ defines a  so-called Schweizer-Wolff measure 
\begin{equation} \label{det11} 
\sigma_p(C) := \frac{\norm{C - \Pi}_p}{\norm{C^+ - \Pi}_p} ~,
\end{equation}
where the copula $C^+$ is defined by $C^+(u, v) = \min \curly{u, v}$ (see \refsecA{} in the Supplementary Material (\cite{Strothmann_supplement.2022})).
The measure $\sigma_1$ was considered in \cite{Schweizer.1981}, \HD{$\sigma_2$ is also known as Blum-Kiefer-Rosenblatt's $R$,} and the general case $p \geq 1$ can be found in Section~5.3.1  of \cite{Nelsen.2006}.
It is easy to see that properties (1.1) and (1.2) hold for $\sigma_p$, and it is well known  that
$\sigma_p(C)=1$   if  and only $Y=f(X)$ for some strictly monotone (but not just measurable) function $f$  \cite[Sect.~5.3.1]{Nelsen.2006}. 
Consequently, $\sigma_p$ does \emph{not} satisfy property (1.3). 
On the other hand, it will be shown in Appendix~\ref{appendix:section2} that the properties (1.1)--(1.3) do hold for  the restriction of  $\sigma_p$   to the set
 $\ir{\C}$. Therefore, the rearranged  Schweizer-Wolff measure 
 $$ R_{\sigma_p}(C) = \frac{\norm{\ir{C} - \Pi}_p}{\norm{C^+ - \Pi}_p} ~ 
 $$
 defines a new dependence measure on $\C$ satisfying all the properties (1.1)--(1.3) on $\C$.
}
\end{example}
 
\begin{example}[Measures of concordance] \label{ex:kappa_measures}
{\rm 
Let $\kappa: \C\to [-1,1]$ be a measure of concordance (see Definition~\refdefconcordance{} in the Supplementary Material (\cite{Strothmann_supplement.2022})).  
Typical examples include Spearman's $\rho$, Kendall's $\tau$, Gini's $\gamma$, and Blomqvist's $\beta$ (see Appendix~\ref{appendix:section2} for a representation of these measures in terms of the copula). 
We will prove in Appendix~\ref{appendix:section2} that  the measures $\rho, \tau$ and $\gamma$  satisfy (1.1)--(1.3) on
the set  $\ir{\C}$ (but not on $\C$); on the other hand, Blomqvist's $\beta$ does not satisfy (1.3) on $\ir{\C}$. 

Consequently, by Theorem \ref{thm:good_measure_R}, 
the  rearranged  Spearman's $\rho$ ($R_{\rho}$),  Kendall's $\tau$ ($R_{\tau}$) and  Gini's $\gamma$ ($R_{\gamma}$)
define dependence measures (different from their original measures) satisfying (1.1)--(1.3) on $\C$.
}
\end{example}

We will now see that, surprisingly, the Schweizer-Wolff measure $\sigma_1$ and 
Spearman's $\rho$ induce  the same rearranged dependence measure.

\begin{proposition} \label{prop:spearman}
We have $R_{\sigma_1} = R_{\rho}$.
\end{proposition}

\begin{remark} \label{rmk:number_dependence_meas}
We point out that there are even uncountably many dependence measures $\mu$ satisfying $R_\mu = R_\rho$. Indeed, pick any function $f:\C\to [0,1]$ being $1$ on $\ir{\C}$ and $0$ outside some neighbourhood of $\ir{\C}$ (apply Urysohn's lemma to the closed convex set $\ir{\C}$), and consider the dependence measures $\mu := f\rho + (1-f) \nu$ where $\nu\neq\rho$. Then $\mu=\rho$ on $\ir{\C}$ so that $R_{\mu}=R_{\rho}$, regardless of the choice of $\nu$.
\end{remark}

\begin{remark}
\HD{
A referee raised the question if 
there exist ``well-known''  dependence measure $\mu \neq r$ such that $R_\mu = R_r$. 
In the following, we will derive a necessary condition for such measures.
Since $r$ is invariant under rearrangement, we have $R_r(C) = r(C)$.
Now suppose $X$ and $Y$ follow a bivariate normal distribution with correlation $p \in [-1, 1]$  and corresponding copula $C_p$.
Since $C_p = \ir{C}_\rho$, it follows that  
for a normal distribution the dependence measure $\mu$ must  satisfy 
\begin{equation*}
    \mu(X, Y)   = \mu(C_p)  
                = \mu(\ir{C}_p)
                = R_\mu(C_p) =R_r (C_p)
                = r(C_p)
                = \frac{3}{\pi} \arcsin\Big (\frac{1 + p^2}{2}\Big ) - \frac{1}{2} ~.
\end{equation*}
We are not aware of any ``well-known''  dependence measure fulfilling this property.
However,  by the same technique as in Remark~\ref{rmk:number_dependence_meas}, it can be shown that there exist infinitely many 
\enquote{dependence measures} $\mu$ on $\mathcal{C}$ such that  $R_{\mu} = R_r$.
Thus, the equivalence class $[r] := \{ \mu \mid R_\mu = R_r \}$ is not a singleton.}
\end{remark}

While a measure of concordance $\kappa$ measures the strength of the monotone association between two random variables, the corresponding rearranged dependence measure $\RDk$ measures the strength of their (directed) functional relationship.
Thus $\kappa$ should always attain smaller values than $\RDk$.
This heuristic  is confirmed by  the next theorem, which applies, in particular, to  Spearman's $\rho$ and Kendall's $\tau$.

\begin{theorem} \label{cor:spearman_positive}
Let $\kappa$ be a measure of concordance satisfying (1.1)--(1.3) on  the set $\ir{\C}$. 
Then
\begin{equation} 
\label{rev1} 
\abs{\kappa(C)} \leq R_{\kappa}(C) 
\end{equation} 
for all $C\in\C$, with equality whenever $C$ is stochastically monotone. 
\end{theorem}

\begin{remark}
\HD{The inequality  \eqref{rev1} connecting the underlying measure $\mu$ and $R_\mu$ can be extended beyond concordance measures.
Whenever the measure $\mu$ is ordered with respect to the pointwise ordering of copulas and fulfils for all random variables $X$ and $Y$
either $\mu(1-X, Y) = -\mu(X, Y)$ or $0 \leq \mu(X, Y)$, then $ \abs{\mu(X, Y)}    \leq R_\mu(X, Y)$.}
\end{remark}

\subsection{Data processing inequality and self-equitability}

Informally, the so-called data processing inequality states that a (random or functional) modification of the input data cannot increase the information contained in the data; see, e.g., \cite{Cover.2006} for an in-depth treatment of the data processing inequality in the context of information theory.

We assume in the following that the dependence measure $\mu$ is monotone on $\ir{\C}$ with respect to the pointwise order, i.e.\ we have
\begin{equation} \label{det12}
C_1 \leq C_2  \implies \mu(C_1) \leq \mu(C_2)    
\end{equation}
for all $C_1,C_2 \in \ir{\C}$. Note that this monotonicity condition holds for  many dependence measures. For example, \eqref{det12} is satisfied  for any  concordance measure (see Definition~\refdefconcordance{} for a precise definition), the Schweizer-Wolff measures $\sigma_p$ in \eqref{det11}  as well as the measures of complete dependence $\zeta_1$ and $r$ 
introduced in Remark \ref{rem1}.

\begin{proposition}[Data processing inequality] \label{prop:dpi}
Assume that the dependence measure $\mu$ satisfies \eqref{det12}, and let $X, Y, Z$ be continuous random variables  such that $Y$ and $Z$ are conditionally independent given $X$.
Then the data processing inequality
\begin{equation*}
	\RDmRV{Z}{Y}	\leq \RDmRV{X}{Y}
\end{equation*}
holds. 
In particular, $\RDmRV{f(X)}{Y}	\leq \RDmRV{X}{Y}$ holds for all\footnote{Note that for $\RDmRV{f(X)}{Y}$ to be well-defined, $f(X)$ needs to be a continuous random variable.} measurable functions $f$.
\end{proposition}

Similar to \cite[Proposition~2.1]{Geenens.2020}, the data processing inequality also immediately yields an asymmetric version of the so-called self-equitability introduced in \cite{Kinney.2014}. 

\begin{corollary} \label{cor:self_equitability}
Assume that $\mu$ satisfies \eqref{det12}.
If $f$ is a measurable function such that $X$ and $Y$ are conditionally independent given $f(X)$, then 
\begin{equation*}
	\RDmRV{f(X)}{Y} = \RDmRV{X}{Y}~.
\end{equation*}
In particular, $\RDmRV{g(X)}{Y} = \RDmRV{X}{Y}$ holds for all measurable bijections $g$. 
\end{corollary}

Intuitively, Corollary~\ref{cor:self_equitability} states that, in a regression model $Y = f(X) + \epsilon$, the dependence measure $\RDm(X,Y)$ depends only on the strength of the noise $\epsilon$ and not on the specific form of $f$. 
A similar idea is illustrated in Figures~3 and 4  of \cite{Junker.2021}.

\subsection{Multivariate rearranged dependence measures} \label{section:multivariate}

In this section we explain how the rearrangement technique can be generalized to a  multivariate setting as follows.
For any measure $m$ on $[0, 1]^d$ and any Borel measurable function $f: [0,1]^d \to \R$, the {\it decreasing rearrangement}  $f^*:  [0,1] \to \mathbb{R} $
 of $f$  is defined by 
 \begin{equation*}
        f^*(t) := \inf \{x\in\mathbb{R} \mid m (\{\boldm{u} \in [0, 1]^d \mid f(\boldm{u}) > x \}) \leq t \} ~.
 \end{equation*}
As in the former case $d=1$, $f^*$ is always a decreasing (univariate) function.

Now, let $(\boldm{X},Y)$ denote a $(d+1)$-dimensional  random vector with  continuous distribution function $F$ and marginal distribution functions $F_{X^1}, \ldots, F_{X^d}$ and $F_Y$.
Using the disintegration approach introduced by \cite{Griessenberger.2022}, for each $(d+1)$-copula $C$ there exists a Markov kernel $K_C: [0, 1]^d \times \mathcal{B}([0, 1]) \rightarrow [0, 1]$ such that
\begin{equation*} \label{eqn:disint_markov}
	C(\boldm{u}, v) = \cInt{K_C(\boldm{s}, [0, v])}{[0, \boldm{u}]}{}{\mu_{C^{1 \cdots d}}(\boldm{s})} ~,
\end{equation*}
where $C^{1 \cdots d}(u_1, \ldots, u_d) := C(u_1, \ldots, u_d, 1)$ denotes the marginal copula with its induced measure $\mu_{C^{1 \cdots d}}$. 
Similar to Theorem~\ref{thm:rearrangement}, 
\begin{align}
\label{det101}
 \ir{C}(u, v) := \cInt{(K_C)^*(s, v)}{0}{u}{s}
 \end{align}
is again a bivariate stochastically increasing copula, where the rearrangement is applied to the measure $m=\mu_{C^{1 \cdots d}}$ and the function $\boldm{s} \mapsto K_C(\boldm{s}, [0, v])$ for every fixed $v \in [0, 1]$. Note that in the case $d=1$ the representation \eqref{det101} reduces to \eqref{det102} since $C^{1}(u_1) := C(u_1, 1)=u_1$ and $K_C({s}, [0, v]) =\partial_1 C(s,v)$ almost everywhere.

Given any bivariate dependence measure $\mu$, the rearranged dependence measure $R_{\mu}(C) := \mu(\ir{C})$ is now a multivariate measure of dependence in the sense that the multivariate versions
\begin{enumerate}
    \item[(M 1.1)] $0 \leq R_\mu(C) \leq 1$
    \item[(M 1.2)] $R_\mu(C) = 0$ if, and only if, $\boldm{X}$ and $Y$ are independent
    \item[(M 1.3)] $R_\mu(C) = 1$ if, and only if, $Y = f(\boldm{X})$ for some measurable function $f: \mathbb{R}^d \rightarrow \mathbb{R}$
\end{enumerate}
 of (1.1)--(1.3) hold for every $(d+1)$-copula $C$.  
 \HD{We finally note that the multivariate rearrangement and the induced rearranged dependence measures enjoy  many properties known from other multivariate measures of complete dependence. 
 For example, the multivariate rearrangement fulfils the information gain inequality    \begin{equation*}
         \ir{C}_{X_1, Y} \leq \ir{C}_{(X_1, X_2), Y} \leq \ldots \leq \ir{C}_{(X_1, \ldots, X_d) , Y} ~,
     \end{equation*}
and this also holds for $R_\mu(C)$ if the dependence measure  $\mu$  is monotone   with respect to the pointwise ordering of copulas.
Moreover, if  $X_2, \ldots, X_d$ and $Y$ are conditionally independent given $X_1$, we have
$    
        \ir{C}_{(X_1, \ldots, X_d), Y} = \ir{C}_{X_1, Y} .
$ 
}

\section{Approximation and estimation} 
 \def\theequation{3.\arabic{equation}}
\setcounter{equation}{0}
\label{sec3}

In general, the computation of the rearrangement of a function, and hence the computation of $\ir{C}$, may be a difficult task.
In this section, we discuss techniques to approximate $\ir{C}$ and $\RDm(C)$ and to estimate the rearranged dependence measure $R_\mu$ from a sample of 
independent and identically distributed observations $(X_1,Y_1) , \ldots , (X_n,Y_n)$. 
In principle, one would like to estimate the copula $C$ through a \enquote{smooth} statistic, say $\hat C_n$, and then apply Definition~\ref{def2.2} to calculate the rearrangement $\ir{\hat C_n}$  and the rearranged dependence measure
\begin{equation} \label{det0}
    \RDm (\hat C_n) = \mu ( \ir{\hat C_n} ).
\end{equation}
While various smooth estimators have been proposed \cite[see][among others]{Fermanian2004,chen2007,omelka2009,GENEST201782}, the simultaneous estimation of the rearrangement poses various difficulties.
We will now propose a simple solution to this problem.

Our approach is based on an  approximation scheme for $\ir{C}$ in the theoretical as well as empirical setting using  the concept of checkerboard copulas, thereby circumventing the need to treat partial derivatives explicitly. 
Checkerboard copulas are an important tool in statistical applications;
for a detailed discussion
we refer, among others, to  \cite{GENEST201782} and \cite{Junker.2021}. 
To be precise let $A = (a_{k\ell} )_{k = 1, \ldots , N_1}^{\ell  = 1, \ldots , N_2} \in \R^{N_1 \times N_2}$ 
denote a  matrix  with  entries $a_{k\ell}$ satisfying  
\begin{equation} \label{det20}
\begin{split}
    a_{k\ell} \geq 0 & \quad\text{ for all } k = 1, \ldots, N_1 \text{ and } ~\ell  = 1, \ldots, N_2~, \\ 
    \sum\limits_{k = 1}^{N_1} a_{k\ell} =  N_1  & \quad\text{ for all } \ell  = 1, \ldots, N_2 ~,\\
    \sum\limits_{\ell = 1}^{N_2} a_{k\ell} = N_2   & \quad\text{ for all } k = 1, \ldots, N_1 ~.
    \end{split}
\end{equation}
Then the function $\CBnn{A} : [0,1]^2 \to [0,1]$ defined by  
\begin{equation}
\label{hd0}
	\CBnn{A}(u, v) := 
	\sum\limits_{k, \ell = 1}^{N_1,N_2} a_{k \ell} \cInt{\indFunc{\left[\frac{k-1}{N_1}, \frac{k}{N_1}\right)}(s)}{0}{u}{s} \cInt{\indFunc{\left[\frac{\ell-1}{N_2}, \frac{\ell}{N_2}\right)}(t)}{0}{v}{t}
\end{equation}
is a copula and called the {\it checkerboard copula of the matrix  $A$. }
For a copula $C$ (see Definition~\refdefcopula{}) its {\it induced checkerboard copula}  is defined as 
\begin{equation} \label{det3}
\CBnn{C} := \CBnn{A_{N_1,N_2}} ~, ~
\end{equation}
where the elements of the doubly stochastic matrix $A_{N_1,N_2}$ are  given by 
\begin{equation} \label{det6}
	(A_{N_1,N_2})_{k\ell} :=  N_1 N_2 \cdot V_C \rbraces{\cbraces{\frac{k-1}{N_1}, \frac{k}{N_1}} \times \cbraces{\frac{\ell-1}{N_2}, \frac{\ell}{N_2}}} ~
\end{equation}
and $V_C (B) $ denotes the measure of the (Borel-)set $B \subset [0,1]^2$ induced by the copula $C$. 

Note that in contrast to most of the literature, we define a (empirical) checkerboard copula also for non-square matrices $A$ satisfying \eqref{det20}. For $N=N_1 = N_2$ the representation  \eqref{hd0}  essentially reduces, up to a scaling factor $N$, to the common definition 
based  on doubly stochastic square matrices  \citep[see][]{GENEST201782,Junker.2021}.
 The consideration of the rectangular case, however, is necessary to address asymmetric dependencies between $X$ and $Y$ 
resp.\ $Y$ and $X$.

We point out that the partial derivatives of the copula $\CBnn{A}$ in \eqref{hd0} are piecewise constant for fixed $v \in [0, 1]$ with
\begin{equation*}
    \partial_1  \CBnn{A } \rbraces{u, \frac{j}{N_2}}   = \frac{1}{N_2} \sum\limits_{\ell = 1}^{j} a_{k \ell} \quad \text{ for } 
        u \in \left[\frac{k-1}{N_1}, \frac{k}{N_1} \right) ~.
\end{equation*}
Thus, the (SI)-rearrangement satisfies $\ir{\CBnn{A}} = \CBnn{A}$ if and only if 
\begin{equation} \label{det5}
	\sum\limits_{j = 1}^\ell a_{k_2 j} \leq \sum\limits_{j = 1}^\ell a_{k_1 j} 
\end{equation}
for all $1\leq \ell \leq N_2$ and all $1 \leq k_1 \leq k_2 \leq N_1$.
In other words, $\ir{\CBnn{A}} = \CBnn{A}$ if and only if the rows of $A$ are ordered with respect to the majorization ordering of vectors \citep[see][]{Marshall.2011}.
This suggests the following Algorithm~\ref{alg1} for calculating the (SI)-rearrangement (as defined in  Definition~\ref{def2.2}) of an arbitrary checkerboard copula.  

\RestyleAlgo{ruled}
\begin{algorithm}[H] 
 \caption{Rearranged checkerboard copula}
\label{alg1}
\medskip
\SetAlgoLined
\KwData{matrix $A \in \mathbb{R}^{N_1 \times N_2}$ with  entries satisfying \eqref{det20}}
\KwResult{(SI)-rearrangement  $\ir{\CBnn{A}}$  of the checkerboard copula $\CBnn{A}$
}
\medskip
\begin{minipage}{.9\linewidth}%
\begin{enumerate} 
\item[(1)]  Calculate $B_k^\ell := \sum_{j = 1}^\ell a_{k j}$ and set $B_k^0 := 0$.
\item[(2)]  For every $\ell = 0, \ldots, N_2$, sort $B^\ell_k$ in a decreasing order and denote the result by $\widetilde{B}^\ell_k$.
\item[(3)] Calculate $\ir{a}_{k \ell}$ iteratively using 
	\begin{equation*}
		\ir{a}_{k \ell} 	
						:= \widetilde{B}^\ell_k - \widetilde{B}^{\ell-1}_k \geq 0  ~.
	\end{equation*}
\item[(4)] Define $\ir{A} :=(\ir{a}_{k\ell} )_{k = 1, \ldots , N_1}^{\ell  = 1, \ldots , N_2}$ and 
\begin{equation*} 
	\ir{\CBnn{A}} := \CBnn{\ir{A}} ~.
\end{equation*}
\end{enumerate}
\end{minipage}
\end{algorithm}

\begin{theorem} \label{thm:CB_algorithm}
For any matrix $A \in \R^{N_1 \times N_2}$ satisfying \eqref{det20},
the function $\ir{\CBnn{A}}$ defined in Algorithm \ref{alg1}
is the (SI)-rearrangement of  the checkerboard copula
$\CBnn{A}$. 
\end{theorem}

We now turn to the estimation of  the population dependence measure $\RDm(C) = \mu(\ir{C})$
from a sample of 
independent and identically distributed observations.
Because there exists in general no analytic expression  for $\RDm(C)$, this   is a challenging 
task and   we proceed in two steps. 
First, note that the population measure $\RDm(C)$ can be approximated  by $\RDm(\CBnn{C})$ using the induced checkerboard copula $\CBnn{C} $ of $C$ defined in \eqref{det3} since
\begin{equation} \label{det7} 
  \ir{\CBnn{C}}  \to \ir{C}
\end{equation}
where $ \ir{\CBnn{C}} $ denotes the  rearrangement of  $\CBnn{C} $.
Secondly, we replace the unknown weights in \eqref{det6} by corresponding estimates to obtain an empirical checkerboard copula, which is then rearranged by Algorithm \ref{alg1}.

We begin with the approximation of $\ir{C} $ by the rearranged induced checkerboard copula.
Since it is well known that the pointwise convergence is unable to capture complete dependence \citep[see][]{Mikusinski.1992}, we consider the finer metrics
\begin{equation*}
     D_p (C_1,C_2) :=  \rbraces{ \cInt{\cInt{\abs{\partial_1 C_1(u, v) - \partial_1 C_2(u, v)}^p}{0}{1}{u}}{0}{1}{v}}^{\frac{1}{p}}
    \end{equation*}
for $1 \leq p < \infty$ introduced in \cite{Trutschnig.2011}. 

\begin{theorem} \label{thm:consistency}
For any copula $C$, the rearranged induced checkerboard copula $ \ir{\CBnn{C}}$ converges to the rearranged copula $\ir{C}$ with respect to $D_p$, i.e.\  
\begin{equation*}
	D_p( \ir{\CBnn{C}}  , \ir{C}) \rightarrow 0 
\end{equation*}
 as $ N_1, N_2  \rightarrow \infty $. 
In particular, $\ir{\CBnn{C}}$ converges uniformly towards $\ir{C}$.
\end{theorem}

In order to carry over the convergence of $\ir{C}_n$ to $\ir{C}$ and establish consistency of the estimator, we require that the underlying dependence measure $\mu$ is continuous on $\ir{\C}$ with respect to pointwise convergence, i.e.\ that
\begin{equation} \label{asmpt:approximation}
    C_n \rightarrow C   \implies \mu(C_n) \rightarrow \mu(C) 
\end{equation}
holds for all copulas $C_n, C\in \ir{\C}$. We point out that most classical measures are continuous in this sense. In fact, any concordance measure (see Definition~\refdefconcordance{}), the Schweizer-Wolff measures $\sigma_p$ in \eqref{det11}, as well as the measures of complete dependence $\zeta_1$ and $r$ in Remark \ref{rem1} fulfil our continuity condition\footnote{For $\zeta_1$ and $r$ this follows from \cite[Prop.~3.6]{Siburg.2021b}.}.

\begin{theorem}
If the dependence measure $\mu$ satisfies  \eqref{asmpt:approximation} then
\begin{equation*}
    \RDm(\CBnn{C})  \rightarrow \RDm(C) ~~~ \text{ as } N_1, N_2 \rightarrow \infty~.
\end{equation*}
\end{theorem}

Next, we consider  a  random sample of independent identically distributed observations $(X_1, Y_1)$, \ldots, $(X_n, Y_n)$. 
Similar to \cite{Li1998} and \cite{Junker.2021}, who considered the  case $N_1=N_2$, we define the empirical checkerboard copula with bandwidth $N_1, N_2 < n$ by 
\begin{align}
\label{hd01}
	 \hat{C}^\#_{N_1,N_2, n}  := C_{N_1,N_2}^\#\big(C_{n,n}^\#(\hat A_n)\big) ~,
\end{align}
where $\hat A_n = (\hat a_{ij})$ is the $n \times n$ permutation matrix defined by 
\begin{equation*}
	\hat a_{ij} := 	\begin{cases}
						1	&\text{ if there exists some $k$ with rank}(X_k) = i \text{ and rank}(Y_k) = j \\
						0	&\text{ else} 
					\end{cases} ~
\end{equation*} 
and $\mathrm{rank}(x_k)$ denotes the rank of  $x_k$  among  $x_{1}, x_{2} , \ldots ,  x_{n}$.
Finally, we define 
\begin{equation}
    \label{det51}
    \hat{R}_\mu := \RDm( \hat{C}^\#_{N_1,N_2, n})
\end{equation}
as an estimator of $R_\mu (C)$, which will be called {\it  rearranged $\mu$-estimate} throughout this paper.  The following result shows strong consistency of $  \hat{R}_\mu $.

\begin{theorem} \label{thm:estimation}
Assume that the dependence measure $\mu$ fulfils the assumption \eqref{asmpt:approximation}, and let $(X_1, Y_1)$, \ldots, $(X_n, Y_n)$ denote independent identically distributed 
random variables  with a  continuous distribution.
If   $N_1 := \lfloor n^{s_1} \rfloor$, $N_2 := \lfloor n^{s_2} \rfloor$  with $s_1,s_2 \in (0, 1/2)$, then
the estimator defined by
\eqref{det51} satisfies 
\begin{align*} 
\hat{R}_\mu
\rightarrow \RDm(C)	\text{ a.s. as } n \rightarrow \infty ~.
\end{align*}
\end{theorem}

\begin{remark}
\HD{  For big data applications the time complexity of the new estimators is   of importance. 
The calculation of the ranks requires $O(n\log(n))$ operations, while, in the absence of ties, the 
empirical checkerboard copula can be obtained by 
$O(n)$ operations. The rearrangement in Algorithm \ref{alg1} can be done by
$$
O(N_1 N_2) + O(N_2 N_1 \log(N_1)) + O(N_1 N_2) = O(n^{s_1} n^{s_2} \log(n^{s_1}) = O(n^{s_1+s_2} \log(n))~
$$
operations,  where $s_1 + s_2 < 1$. Naturally, the time complexity to compute the dependence measure depends on the underlying measure but oftentimes requires  $O(N_1N_2) = O(n^{s_1 + s_2})$
operations. As a result, for fixed $s_1,s_2$
the time complexity is of order $O(n\log(n))$, which coincides with the complexity of  Chatterjee's estimator.
}
\end{remark} 

\section{Finite sample properties} \label{sec4}

For a good performance of the estimate  $\hat{C}^\#_{N_1,N_2, n}  $,
an appropriate  choice of the bandwidths
$N_1,N_2$ will be crucial. These tuning parameters   depend
sensitively on  the form of the underlying unknown copula, and 
for  the finite sample illustrations presented below, we  will use  the following  cross validation principle, which is adapted from density estimation \citep[see, for example][]{rice1984}.

Recall the definition of the empirical checkerboard copula $ \hat{C}^\#_{ N_1,N_2,n }$, and denote its corresponding density by 
\begin{equation} \label{det10} 
\hat{c}_{ N_1,N_2,n } (u,v) := { \frac{\partial^2}{\partial u \partial v}} 
\hat{C}^\#_{ N_1,N_2,n } (u,v) ~.
\end{equation}
We define
$$
{\rm CV}({ N_1,N_2,n })  := \cInt{\cInt{\hat{c}_{ N_1,N_2,n }^2 (u,v)}{0}{1}{u}}{0}{1}{v}
- {\frac{2}{n}}  \sum_{i=1}^n 
\hat{c}_{ N_1,N_2,n-1 }^{-i}  (\hat U_i,\hat V_i)~,
$$
where $\hat{c}_{ N_1,N_2,n -1}^{-i} $ denotes  the estimator in \eqref{det10} calculated from the data  
$$
(X_1,Y_1), \ldots , (X_{i-1},Y_{i-1}),  (X_{i+1},Y_{i+ 1}), \ldots ,  (X_n,Y_n)~
$$
and $\hat U_i = {\frac{1}{n+1}}\sum_{j=1}^n 
I\{ X_j \leq X_i \} $
and $\hat V_i= {\frac{1}{n+1}}\sum_{j=1}^n 
I\{ Y_j \leq Y_i \}  $  are the normalized ranks of $X_i$ and $Y_i$ among  $X_1, \ldots , X_n$ and $Y_1, \ldots , Y_n$, respectively.
The data adaptive choice of the parameters
$N_1$ and $N_2$ is defined as the minimizer of ${\rm CV}(N_1,N_2, n) $ with respect to $N_1,N_2\in \{\lfloor n^{1/4} \rfloor  , \ldots , \lfloor n^{1/2} \rfloor  \}$.
\HD{
Note that, although consistency of the empirical checkerboard copula  holds for 
$\lfloor n^{s_i} \rfloor$ with $s_i \in (0,1/2) $ 
(see \cite{Junker.2021}),
we perform cross-validation  only for   $s_1, s_2 \in (1/4 , 1/2)$.
On the one hand, this saves computational time. 
On the other hand,  for $ s < 1/4$, the number of grid divisions is extremely small yielding almost no discernible influence on the outcome of the cross-validation procedure.  Moreover, in  cases, where  the set of possible bandwidths  is very  large, we   
calculate the minimizer on  the set $\{\lfloor n^{1/4} \rfloor  ,  \lfloor n^{1/4} \rfloor +2 , 
\ldots , \lfloor n^{1/2} \rfloor  \}$ in order to save further computational time.
 }

\subsection{Simulation study} \label{section:simulations}  

In this section, we present results from a simulation study   investigating  the performance of the   estimator $\hat{R}_\mu$ defined in \eqref{det51}. All simulations have been conducted using the statistical software \enquote{R} \citep[see][]{RCore.2021}
and are based on $1000$ replications in each scenario.
\HD{The code used in the simulation study can be found at \href{https://github.com/ChristopherStrothmann/RDM}{https://github.com/ChristopherStrothmann/RDM}.}
Therein, the package \enquote{qad}  \citep[see][]{qad.2022} was used in a slightly adapted form to calculate the  matrix $\hat A_n$, which is required for the definition of 
the empirical checkerboard copula in \eqref{hd01}.
As  sample   sizes  we considered  $n = 50, 100, 500$ and $1000$ and 
$N_1, N_2$ were chosen by the cross validation procedure described at the beginning of this section.

\subsubsection{Stochastically increasing  distributions}

We begin with a study  of the properties of the estimator 
\eqref{det51} in the rather special case where the underlying copula is 
stochastically increasing.  The corresponding samples  have been generated using the package \enquote{copula} 
\cite[see][]{Hofert.2020}. 
As for stochastically monotone copulas we have 
$R_\mu = \mu $,  we can calculate the dependence measure explicitly, and it is also reasonable to compare the 
new estimator $\hat{R}_\mu $ with commonly used  estimators of 
$\mu$.
\HD{The R-packages \enquote{XICOR} was used to estimate Chatterjee's coefficient $\hat{\xi}$ of $r$ (see \cite{Chatterjee.2020}) and \enquote{qad} (see \cite{qad.2022}) was used to estimate $\hat{\zeta}_1$ of Trutschnig's $\zeta_1$.}

The first two scenarios correspond to a $2$-dimensional (centred) normal distribution with correlation matrix 
\begin{equation}
 \label{det54}   
R = \begin{pmatrix}
1 \; &p \\ p &1
\end{pmatrix}
\end{equation}
where $p=0.25$ and $p=0.75$, respectively. 
Since for $p >0$, the corresponding copula, say $C_p$, is stochastically increasing, the rearranged Spearman's $\rho$ equals
\begin{equation*}
	R_\rho(C_p) 	= \RD(X, Y)	
				= \rho(X, Y)
				= \frac{6}{\pi} \arcsin \rbraces{\frac{{p}}{2}}
				~~~(p\geq 0)~,~~
\end{equation*}
while the rearranged Kendall's $\tau$ equals
\begin{equation*}
	\RDtau(C_r) 	= \RDtau(X, Y)	
				= \tau(X, Y)
				= \frac{2}{\pi} \arcsin \rbraces{{p}} 	~~~(p\geq 0)~.
\end{equation*}

The third example of a stochastically increasing  copulas is a member of both the Archimedean and extreme-value copula families, which are widely applied, both theoretically as well as empirically. 
More precisely, we consider a Gumbel copula defined by
\begin{equation*}  
    C_\theta^G (u, v) := \exp \rbraces{- \rbraces{(-\log u)^\theta + (-\log v)^\theta}^{{1}/{\theta}}} ~,
\end{equation*}
where   $\theta >1 $ denotes a  parameter. 
Since the Gumbel copula is an extreme-value copula, it is stochastically increasing, where the rearranged Spearman's $\rho$ and Kendall's $\tau$ are given by  
\begin{align*}
   R_\rho(C_{\theta}^G) =  {\rho(X, Y)} = 
   12 \cInt{\frac{1}{\rbraces{1+(t^\theta + (1-t)^\theta)^{1/\theta}}^2}}{0}{1}{t} - 3 \quad \text{ and } \quad
    \tau(C_{\theta}^G) = {\tau(X, Y)}=  \frac{\theta - 1}{\theta} ~.
\end{align*}

In Figure~\ref{fig:introduction_dependence1},  we show scatter plots of data generated from the two Gaussian
copulas ($p=0.25$, $p=0.75$)
and the Gumbel  copula  ($\theta = 3$), where the sample size is $n=500$. In  the upper part of  Table \ref{table:si_spearman},
 we present the simulated mean and standard deviation of the rearranged 
estimate $\hat R_\mu $, where $\mu $  is either Spearman's $\rho$ (left part) or  Kendall's $\tau$
(right part). Due to $R_\mu (C)= \mu (C)$ for the three scenarios, the  commonly used 
Spearman's rank correlation coefficient $\hat \rho$ and 
Kendall's rank correlation coefficient  $\hat \tau$ can also be used to estimate
$R_\rho (C) $ and $R_\tau (C) $, respectively. The  corresponding results  for these estimates are displayed in  Table \ref{table:si_spearman}
as well
(of course, in practice it is not known if the underlying copula is stochastically increasing).

\begin{table}
\begin{tabular}{c|c | ccc | ccc}
\toprule
\multirow{2}{*}{copula} &\multirow{2}{*}{$n$} &\multicolumn{3}{c}{Spearman's $\rho$ }  &\multicolumn{3}{|c}{Kendall's $\tau$} \\
 & &$\RDrho$ &$\hat{R}_\rho$ &$\abs{\hat{\rho}}$ &$\RDtau$ &$\hat{R}_\tau$ &$\abs{\hat{\tau}}$\\ 
\midrule
\multirow{4}{*}{$C_{0.25}$} &50&\multirow{4}{*}{0.239} &0.276 (0.132) &0.246 (0.129)&\multirow{4}{*}{0.161} &0.185 (0.090) &0.169 (0.090)\\ 
 &100 &  &0.263 (0.093) &0.236 (0.096) &  &0.176 (0.063) &0.160 (0.066)\\ 
 &500 &  &0.224 (0.043) &0.240 (0.043) &  &0.150 (0.029) &0.162 (0.029)\\ 
 &1000 &  &0.226 (0.030) &0.239 (0.029) &  &0.151 (0.020) &0.160 (0.020)\\ 
\midrule
\multirow{4}{*}{$C_{0.75}$} &50&\multirow{4}{*}{0.734} &0.669 (0.094) &0.721 (0.075)&\multirow{4}{*}{0.540} &0.473 (0.079) &0.538 (0.068)\\ 
 &100 &  &0.694 (0.063) &0.727 (0.051) &  &0.496 (0.055) &0.539 (0.047)\\ 
 &500 &  &0.714 (0.025) &0.732 (0.023) &  &0.517 (0.023) &0.539 (0.020)\\ 
 &1000 &  &0.723 (0.017) &0.734 (0.015) &  &0.527 (0.015) &0.540 (0.014)\\ 
\midrule
\multirow{4}{*}{$C_{3}^G$} &50 &\multirow{4}{*}{0.848} &0.803 (0.057) &0.839 (0.050)&\multirow{4}{*}{0.667} &0.599 (0.058) &0.668 (0.055)\\ 
 &100 &  &0.826 (0.040) &0.844 (0.037) &  &0.628 (0.044) &0.667 (0.041)\\ 
 &500 &  &0.844 (0.016) &0.848 (0.015) &  &0.653 (0.019) &0.666 (0.017)\\ 
 &1000 &  &0.847 (0.011) &0.848 (0.010) &  &0.659 (0.012) &0.666 (0.012)\\ 
 \midrule
 \multirow{2}{*}{copula} &\multirow{2}{*}{$n$} &\multicolumn{3}{c}{$r$} &\multicolumn{3}{c}{$\zeta_1$} \\
 & &$r$ &$\hat{R}_{r}$ &$\hat{\xi}$ &$\zeta_1$ &$\hat{R}_{\zeta_1}$ &$\hat{\zeta_1}$\\ 
\midrule
\multirow{4}{*}{$C_{0.25}$} &50&\multirow{4}{*}{0.030} &0.070 (0.054) &0.076 (0.058)&\multirow{4}{*}{0.170} &0.236 (0.102) &0.323 (0.065)\\ 
 &100 &  &0.060 (0.036) &0.058 (0.043) &  &0.222 (0.071) &0.307 (0.050)\\ 
 &500 &  &0.035 (0.013) &0.038 (0.025) &  &0.173 (0.032) &0.249 (0.024)\\ 
 &1000 &  &0.035 (0.009) &0.035 (0.020) &  &0.172 (0.024) &0.231 (0.019)\\ 
\midrule
\multirow{4}{*}{$C_{0.75}$} &50&\multirow{4}{*}{0.360} &0.333 (0.087) &0.330 (0.092)&\multirow{4}{*}{0.560} &0.557 (0.077) &0.580 (0.065)\\ 
 &100 &  &0.342 (0.060) &0.344 (0.066) &  &0.560 (0.054) &0.589 (0.047)\\ 
 &500 &  &0.347 (0.026) &0.355 (0.030) &  &0.560 (0.025) &0.583 (0.022)\\ 
 &1000 &  &0.348 (0.018) &0.355 (0.021) &  &0.560 (0.017) &0.577 (0.016)\\ 
\midrule
\multirow{4}{*}{$C_3^G$} &50&\multirow{4}{*}{0.520} &0.483 (0.072) &0.490 (0.083)&\multirow{4}{*}{0.690} &0.676 (0.059) &0.682 (0.053)\\ 
 &100 &  &0.502 (0.055) &0.506 (0.059) &  &0.686 (0.045) &0.696 (0.041)\\ 
 &500 &  &0.513 (0.024) &0.516 (0.025) &  &0.690 (0.019) &0.700 (0.018)\\ 
 &1000 &  &0.518 (0.016) &0.519 (0.019) &  &0.693 (0.013) &0.700 (0.013)\\ 
\midrule
\end{tabular}
\caption{ \it \HD{Simulated mean and standard deviation of various estimators of dependence measures. Upper part:
rearranged   Spearman's $\rho$ estimate $\hat R_\rho$, (absolute) Spearman's rank correlation coefficient $|\hat \rho |$ (left part),  
rearranged   Kendall's $\tau$ estimate $\hat R_\tau$ and (absolute) Kendall's rank correlation coefficient $|\hat \tau |$ (right part). 
Lower part: Siburg-Dette-Stoimenov $r$ (i.e. Chatterjee's $\xi$)
estimate $\hat{R}_r$, Chatterjee's estimator $\hat{\xi}$ (left part),   Trutschnig $\zeta_1$ estimate $\hat{R}_{\zeta_1}$ and $\hat{\zeta_1}$ (right part).  The distribution of $(X,Y)$ is given by a centred normal with correlation matrix
\eqref{det54} with copula $C_p$ and by a Gumbel copula $C_{3}^{G}$.}
}
\label{table:si_spearman}
\end{table}

We observe a reasonable behaviour of all rearranged estimates, which improves with increasing sample size. 
In general, there are only minor differences between the rearranged estimates
$\hat R_\rho$, $\hat R_\tau$ and the  non-rearranged estimates  $\hat\rho$, $\hat  \tau$, which are mainly caused by a slightly smaller bias
of the  non-rearranged estimates. 
For the Gaussian copula with correlation $0.25$, the rearranged estimates $\hat R_\rho$ and $\hat R_\tau$ slightly overestimate their population version $ R_\rho$ and $R_\tau$ if the sample size is $n=50$ or $100$. 
For all other scenarios, we observe an underestimation.
\HD{ The lower part of Table \ref{table:si_spearman}  shows some results for the complete dependence   measures $r$  and $\zeta_1$.
We observe  that  the estimator $\hat{R}_r$ and Chatterjee's estimator 
$\hat{\xi}$ behave very similar. On the other  the estimator $\hat{R}_{\zeta_1}$ seems to have a smaller bias than 
$\hat{\zeta_1}$ but at the cost of a larger variance. }

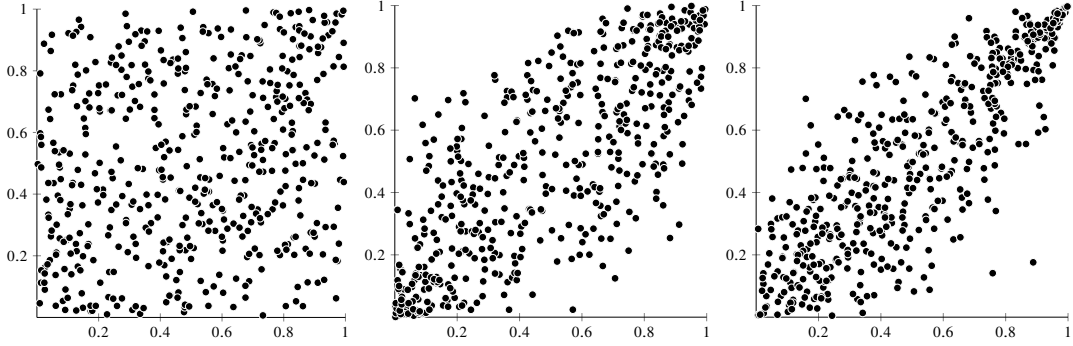
\begin{figure}
\begin{subfigure}[b]{0.33\linewidth}
\centering
\resizebox{\linewidth}{!}{
\begin{tikzpicture}
	\begin{axis}[xmin=0, ymin=0, xmax = 1, ymax=1, axis equal image, axis lines=middle, inner axis line style={-}]
		\addplot[only marks, mark options={draw=white, scale = 1}] table [header=false,x index=0,y index=1, col sep=comma] {scatterplot/normal0.25.csv}; 
	\end{axis}
\end{tikzpicture}}
\end{subfigure}\hfill%
\begin{subfigure}[b]{0.33\linewidth}
\resizebox{\linewidth}{!}{\begin{tikzpicture}
	\begin{axis}[xmin=0, ymin=0, xmax = 1, ymax=1, axis equal image, axis lines=middle, inner axis line style={-}]
		\addplot[only marks, mark options={draw=white, scale = 1}] table [header=false,x index=0,y index=1, col sep=comma] {scatterplot/normal0.75.csv}; 
	\end{axis}
\end{tikzpicture}}
\end{subfigure}\hfill%
\begin{subfigure}[b]{0.33\linewidth}
\resizebox{\linewidth}{!}{\begin{tikzpicture}
	\begin{axis}[xmin=0, ymin=0, xmax = 1, ymax=1, axis equal image, axis lines=middle, inner axis line style={-}]
		\addplot[only marks, mark options={draw=white, scale = 1}] table [header=false,x index=0,y index=1, col sep=comma] {scatterplot/gumbel3.csv}; 
	\end{axis}
\end{tikzpicture}}
\end{subfigure}
\caption{\it Scatter plots of data (sample size $n=500$) from the Gaussian copula with correlation $r=0.25$ (left panel), $r=0.75$ (middle panel) and the Gumbel copula with parameter $\theta=3$ (right panel).}
\label{fig:introduction_dependence1}
\end{figure}

\bigskip

\subsubsection{A family of non-stochastically monotone distributions}

In this section, we consider the more common situation where $R_\mu \not = \mu $. 
To generate data from a family of $2$-dimensional distributions with different degrees of dependence, let $X\sim U(0, 1) $ denote a uniformly (on the interval $[0,1]$) distributed 
random variable and $Z \sim \mathcal{N}(0, 1)$  a standard normal 
distributed  random variable such that $X$ and $Z$ are independent.
We consider the regression model
\begin{equation}
\label{det53}
    Y := (X-1/2)^2 + \sigma Z~,
\end{equation}
where $\sigma$ is a non-negative constant. \HD{Note that the correlation between $X$ and $Y$ is $0$, by construction and that 
a similar model has been studied in \cite{Chatterjee.2020}.}
Model \eqref{det53} contains  perfect functional dependence  of $X$ and $Y$ (for $\sigma=0$) 
and independence in the limit for  $\sigma \to \infty$.
The corresponding scatter plots from   $n=500$  independent  observations 
according to model \eqref{det53}
with  $\sigma = 0$, $0.1$ and $0.3$ are displayed in Figure \ref{fig2}, 
while the upper part of Table \ref{table:general_spearman} shows the simulated mean and standard deviation 
of the estimates $\hat R_\rho $ (for the rearranged Spearman's $\rho$)
and $\hat R_\tau $ (for the rearranged Kendall's $\tau$).
For $\sigma > 0 $ the ``true'' values of $R_\rho$ and $R_\tau$ have been obtained by 
simulation using a sample of size
$n=1 000 000$ and bandwidths $N_1=N_2 = \lfloor n^{0.45} \rfloor$. 
The 
empirical results confirm the consistency statement in Theorem \ref{thm:estimation}.
In the table, we also display the 
simulated mean of the non-rearranged estimators $|\hat \rho| $
and $|\hat \tau | $, which 
do not yield reasonable results.
\HD{In the lower part of Table \ref{table:general_spearman}  we show
again some results for the complete dependence   measures $r$  and $\zeta_1$.
For $\sigma =0.1$  and $\sigma =0.3$ 
the differences between the  estimators $\hat{R}_r$ and 
$\hat{\xi}$ are again very small. However for complete dependence ($\sigma=0$) 
Chatterjee's estimator yields a better performance.  In this case the  estimators $\hat{R}_{\zeta_1}$  and  $\hat{\zeta_1}$ 
exhibit a similar behaviour, while with increasing $\sigma$ the estimator $\hat{\zeta_1}$  has a larger bias than 
$\hat{R}_{\zeta_1}$   (but a slightly
smaller variance).}

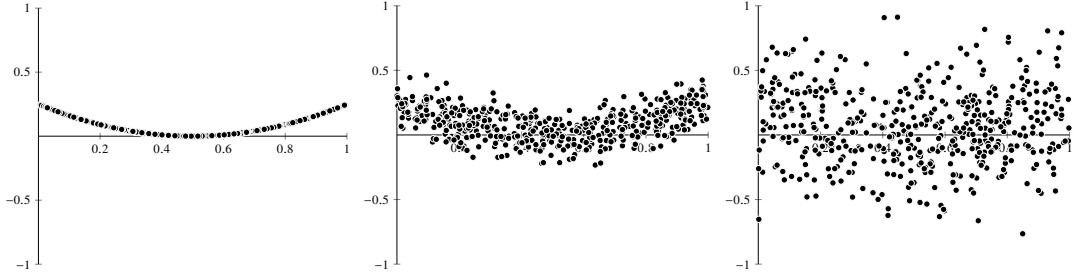
\begin{figure}
\begin{subfigure}[b]{0.33\linewidth}
\centering
\resizebox{\linewidth}{!}{
\begin{tikzpicture}
	\begin{axis}[xmin=0, xmax = 1, ymin=-1, ymax=1, axis lines=middle, inner axis line style={-}]
		\addplot[only marks, mark options={draw=white, scale = 1}] table [header=false,x index=0,y index=1, col sep=comma] {scatterplot/squared0.csv}; 
	\end{axis}
\end{tikzpicture}}
\end{subfigure}\hfill%
\begin{subfigure}[b]{0.33\linewidth}
\resizebox{\linewidth}{!}{\begin{tikzpicture}
	\begin{axis}[xmin=0, xmax = 1, ymin=-1, ymax=1, axis lines=middle, inner axis line style={-}]
		\addplot[only marks, mark options={draw=white, scale = 1}] table [header=false,x index=0,y index=1, col sep=comma] {scatterplot/squared0.1.csv}; 
	\end{axis}
\end{tikzpicture}}
\end{subfigure}\hfill%
\begin{subfigure}[b]{0.33\linewidth}
\resizebox{\linewidth}{!}{\begin{tikzpicture}
	\begin{axis}[xmin=0, xmax = 1, ymin=-1, ymax=1, axis lines=middle, inner axis line style={-}]
		\addplot[only marks, mark options={draw=white, scale = 1}] table [header=false,x index=0,y index=1, col sep=comma] {scatterplot/squared0.3.csv}; 
	\end{axis}
\end{tikzpicture}}
\end{subfigure}
\caption{ \it Scatter plots of a sample of $n= 500$ observations from model \eqref{det53}.
Left panel: $\sigma =  0$; middle  panel: $\sigma =  0.1$;  right panel: $\sigma =  0.3$.}
\label{fig2}
\end{figure}

\begin{table}
\begin{tabular}{c|c | ccc | ccc}
\toprule
\multirow{2}{*}{$\sigma$} &\multirow{2}{*}{$n$} &\multicolumn{3}{c}{Spearman's $\rho$} &\multicolumn{3}{|c}{Kendall's $\tau$} \\
 & &$\RDrho$ &$\hat{R}_\rho$ &$\abs{\hat{\rho}}$ &$\RDtau$ &$\hat{R}_\tau$ &$\abs{\hat{\tau}}$\\ 
\midrule
\multirow{4}{*}{0} &50&\multirow{4}{*}{1} &0.918 (0.012) &0.155 (0.120)&\multirow{4}{*}{1} &0.732 (0.021) &0.133 (0.102)\\ 
 &100 &  &0.960 (0.005) &0.105 (0.078) &  &0.810 (0.013) &0.092 (0.069)\\ 
 &500 &  &0.992 (0.000) &0.048 (0.036) &  &0.914 (0.002) &0.041 (0.031)\\ 
 &1000 &  &0.996 (0.000) &0.032 (0.025) &  &0.939 (0.001) &0.028 (0.022)\\ 
\midrule
\multirow{4}{*}{0.1} &50&\multirow{4}{*}{0.580} &0.530 (0.116) &0.131 (0.094)&\multirow{4}{*}{0.404} &0.362 (0.085) &0.092 (0.066)\\ 
 &100 &  &0.550 (0.081) &0.091 (0.068) &  &0.378 (0.060) &0.063 (0.047)\\ 
 &500 &  &0.553 (0.035) &0.042 (0.031) &  &0.381 (0.026) &0.029 (0.021)\\ 
 &1000 &  &0.559 (0.024) &0.030 (0.022) &  &0.386 (0.018) &0.020 (0.015)\\ 
\midrule
\multirow{4}{*}{0.3} &50&\multirow{4}{*}{0.232} &0.255 (0.143) &0.113 (0.085)&\multirow{4}{*}{0.155} &0.171 (0.096) &0.078 (0.058)\\ 
 &100 &  &0.258 (0.098) &0.081 (0.059) &  &0.173 (0.066) &0.055 (0.040)\\ 
 &500 &  &0.216 (0.047) &0.037 (0.027) &  &0.145 (0.032) &0.025 (0.018)\\ 
 &1000 &  &0.217 (0.033) &0.026 (0.020) &  &0.146 (0.022) &0.017 (0.013)\\ 
\midrule
\multirow{2}{*}{$\sigma$} &\multirow{2}{*}{$n$} &\multicolumn{3}{c}{$r$} &\multicolumn{3}{c}{$\zeta_1$} \\
 & &$r$ &$\hat{R}_{r}$ &$\hat{\xi}$ &$\zeta_1$ &$\hat{R}_{\zeta_1}$ &$\hat{\zeta_1}$\\ 
\midrule
\multirow{4}{*}{0} &50&\multirow{4}{*}{1} &0.667 (0.026) &0.885 (0.002)&\multirow{4}{*}{1} &0.817 (0.018) &0.823 (0.015)\\ 
 &100 &  &0.765 (0.015) &0.941 (0) &  &0.875 (0.010) &0.878 (0.009)\\ 
 &500 &  &0.892 (0.004) &0.988 (0) &  &0.945 (0.002) &0.945 (0.002)\\ 
 &1000 &  &0.923 (0.002) &0.994 (0) &  &0.961 (0.001) &0.961 (0.001)\\ 
\midrule
\multirow{4}{*}{0.1} &50&\multirow{4}{*}{0.22} &0.208 (0.080) &0.202 (0.094)&\multirow{4}{*}{0.46} &0.43 (0.098) &0.481 (0.073)\\ 
 &100 &  &0.215 (0.059) &0.209 (0.071) &  &0.448 (0.066) &0.487 (0.054)\\ 
 &500 &  &0.206 (0.025) &0.215 (0.032) &  &0.44 (0.030) &0.478 (0.026)\\ 
 &1000 &  &0.208 (0.017) &0.217 (0.023) &  &0.444 (0.021) &0.474 (0.018)\\ 
\midrule
\multirow{4}{*}{0.3} &50&\multirow{4}{*}{0.04} &0.068 (0.057) &0.074 (0.057)&\multirow{4}{*}{0.19} &0.223 (0.108) &0.322 (0.066)\\ 
 &100 &  &0.055 (0.036) &0.058 (0.043) &  &0.212 (0.074) &0.3 (0.050)\\ 
 &500 &  &0.033 (0.013) &0.037 (0.024) &  &0.178 (0.035) &0.249 (0.025)\\ 
 &1000 &  &0.032 (0.010) &0.033 (0.020) &  &0.176 (0.026) &0.23 (0.020)\\ 
\midrule
\end{tabular}
\caption{\it  
\HD{ 
Simulated mean and standard deviation of various estimators of dependence measures. Upper part: 
rearranged   Spearman's $\rho$ estimate $\hat R_\rho$, (absolute) Spearman's rank correlation coefficient $|\hat \rho |$ (left part),  
rearranged   Kendall's $\tau$ estimate $\hat R_\tau$, (absolute) Kendall's rank correlation coefficient $|\hat \tau |$ (right part). 
Lower part:  Siburg-Dette-Stoimenov $r$ (i.e. Chatterjee's $\xi$)
estimate $\hat{R}_r$, Chatterjee's estimator $\hat{\xi}$ (left part), Trutschnig $\zeta_1$ estimate $\hat{R}_{\zeta_1}$ and $\hat{\zeta_1}$ (right part). The distribution of $(X,Y)$ is given by model 
\eqref{det53}. }}
\label{table:general_spearman}
\end{table}

\bigskip




\subsection{Power analysis} \label{section:power-analysis}

\HD{In this section we compare 
different dependence measures 
when they are applied for independence testing, that is
\begin{equation*}
    H_0: X ~{\rm and  } ~ Y  ~~ {\rm are~ independent} 
\end{equation*}
For modeling the dependence structure we consider a Gaussian copula $C_p$ with $p \in [0, 1]$ and the copula corresponding  to the non-stochastically monotone distribution in model \eqref{det53} with parameter $1/\sigma $. 
All results are based
on  a sample of size  $n = 200$ and 
$2000$ simulation runs are used to calculate the rejection probabilities.
\\
We consider tests based on estimators
for the rearranged 
Spearman's $\rho$ and rearranged Kendall's $\tau$. 
 As benchmark  we also study the  tests based on 
${\zeta}_1$  (see \cite{Trutschnig.2011, Junker.2021})
and $r$  (see \cite{Chatterjee.2020,Dette.2012}).
For the tests based on $R_\rho$, $R_\tau$ and $\zeta_1$ we use a permutation test  with $1000$ permutations. For the test based on $r$ we use  Chatterjee's  test  with quantiles from the asymptotic distribution.  The simulated rejection probabilities  are displayed in Figure 
 \ref{fig:power-analysis-200-model}  for the Gaussian copula (left panel)  and   model \eqref{det53} (right panel). 
 Note that the case of independence corresponds to the choice $p=0$ and $1/\sigma \rightarrow 0 $,
 while complete dependence is obtained for  
 $p=1$ and $1/\sigma =  \infty$.
 In this case, the  nominal level is well approximated by all $4$  tests. 
With increasing $p$  and  increasing  $1/ \sigma $  we model more dependence and we can study the power of the different tests.  We observe 
no differences  in the  power of the tests based on $R_\rho$, $R_\tau$ and $\zeta_1$. However, all tests outperform Chatterjee's test.
}

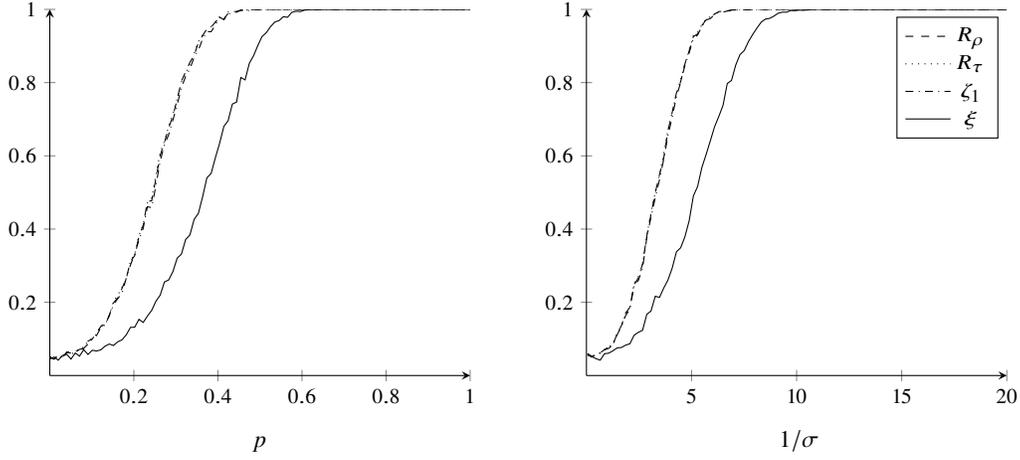
\begin{figure}
\captionsetup[subfigure]{justification=centering, labelformat=empty}
\begin{subfigure}[b]{0.5\linewidth}
\centering
\begin{tikzpicture}
\begin{axis}[axis lines = middle, ymin=0, ymax = 1, xmin=0, xmax=1, width=\linewidth, height=0.9\linewidth]
    \addplot[mark options={}, dashed] table [x index=0,y index=1, col sep=comma] {Power-Analysis/gaussian-200.csv};   
    \addplot[mark options={}, dotted] table [x index=0,y index=2, col sep=comma] {Power-Analysis/gaussian-200.csv};   
    \addplot[mark options={}, dashdotted] table [x index=0,y index=3, col sep=comma] {Power-Analysis/gaussian-200.csv};  
    \addplot[mark options={}, solid] table [x index=0,y index=5, col sep=comma] {Power-Analysis/gaussian-200.csv};  
\end{axis}
\end{tikzpicture}
\caption{~~~~~ $p$} 
\end{subfigure}\hfill
\begin{subfigure}[b]{0.5\linewidth}
\centering
\begin{tikzpicture}
\begin{axis}[axis lines = middle, ymin=0, ymax = 1, xmin=0, xmax=20, width=\linewidth, height=.9\linewidth]
    \addplot[mark options={}, dashed] table [x index=0,y index=1, col sep=comma] {Power-Analysis/squared-200.csv};   
    \addplot[mark options={}, dotted] table [x index=0,y index=2, col sep=comma] {Power-Analysis/squared-200.csv};   
    \addplot[mark options={}, dashdotted] table [x index=0,y index=3, col sep=comma] {Power-Analysis/squared-200.csv};  
    \addplot[mark options={}, solid] table [x index=0,y index=5, col sep=comma] {Power-Analysis/squared-200.csv};  
    \addlegendentry{$R_\rho$}
	\addlegendentry{$R_\tau$}
	\addlegendentry{$\zeta_1$}
	\addlegendentry{$\xi$}
\end{axis}
\end{tikzpicture}
\caption{~~~~~ $1/\sigma $} 
\end{subfigure}
  \caption{\it 
  \HD{Simulated rejection probabilities of  permutation tests  based on 
  different dependence measures. 
  Left panel: the Gaussian copula 
  with correlation $p \in [0,1]$. 
  Right panel:
the non-stochastically monotone family 
defined by  \eqref{det53}, where the 
the parameter $\sigma$ satsifies $1/\sigma \in [0.05, 20]$. The sample size is $n=200$
and the significance level is $\alpha = 0.05\%$.}
\label{fig:power-analysis-200-model}
}
\end{figure}

\subsection{Data example}
\label{sec42} 

In this section we briefly revisit a  data example which  was investigated by \cite{Chatterjee.2020}
to study the performance of his  correlation coefficient in the analysis of  yeast gene expression data.  The data 
consists of  the  expressions of $6223$ yeast genes and  was originally analyzed 
by  \cite{spellman1998} who  tried 
to identify genes whose transcript levels oscillate during the cell cycle. 
For each gene, the gene expression was 
observed at $23$ time points. Because the number of genes is large, visual inspection is not possible
and  \cite{Reshef2011} proposed to use the MIC and MINE correlation coefficient to analyze the data.
\cite{Chatterjee.2020} compared the performance of his correlation coefficient with these measures and 
demonstrated some advantages of his approach.
We will now provide a brief 
illustration  analyzing this type of data
with a rearranged dependence measure to demonstrate the ability of our approach to 
also detect non-monotone dependencies.
We begin with  an  analysis  
of the rearranged  Spearman's rank coefficient $\hat R_\rho$. After that, we provide a very brief comparison of $\hat R_\rho$ with Chatterjee's correlation coefficient.

\begin{figure}
\captionsetup[subfigure]{justification=centering, labelformat=empty}
\begin{subfigure}[b]{0.5\linewidth}
\centering
\begin{tikzpicture}
\begin{axis}[tick style={draw=none}, xticklabels={}, yticklabels={}, width=\linewidth, height=0.63\linewidth]
    \addplot[only marks, mark options={draw=white, scale = .75}] table [header=false,x index=0,y index=1, col sep=comma] {Genes-10000/RDM-Spearman_005_005_1_YBL003C.csv}; 
    \addplot[mark options={}, dashed] table [header=false,x index=0,y index=1, col sep=comma] {Genes-10000/RDM-Spearman_005_005_1_YBL003C_line.csv};   
\end{axis}
\end{tikzpicture}
\caption{YBL003C} 
\end{subfigure}\hfill
\begin{subfigure}[b]{0.5\linewidth}
\centering
\begin{tikzpicture}
\begin{axis}[tick style={draw=none}, xticklabels={}, yticklabels={}, width=\linewidth, height=0.63\linewidth]
    \addplot[only marks, mark options={draw=white, scale = .75}] table [header=false,x index=0,y index=1, col sep=comma] {Genes-10000/RDM-Spearman_005_005_2_YBL009W.csv}; 
    \addplot[mark options={}, dashed] table [header=false,x index=0,y index=1, col sep=comma] {Genes-10000/RDM-Spearman_005_005_2_YBL009W_line.csv};   
\end{axis}
\end{tikzpicture}
\caption{YBL009W}
\end{subfigure}
\begin{subfigure}[b]{0.5\linewidth}
\centering
\begin{tikzpicture}
\begin{axis}[tick style={draw=none}, xticklabels={}, yticklabels={}, width=\linewidth, height=0.63\linewidth]
    \addplot[only marks, mark options={draw=white, scale = .75}] table [header=false,x index=0,y index=1, col sep=comma] {Genes-10000/RDM-Spearman_005_005_3_YBR202W.csv}; 
    \addplot[mark options={}, dashed] table [header=false,x index=0,y index=1, col sep=comma] {Genes-10000/RDM-Spearman_005_005_3_YBR202W_line.csv};   
\end{axis}
\end{tikzpicture}
\caption{YBR202W}
\end{subfigure}\hfill
\begin{subfigure}[b]{0.5\linewidth}
\centering
\begin{tikzpicture}
\begin{axis}[tick style={draw=none}, xticklabels={}, yticklabels={}, width=\linewidth, height=0.63\linewidth]
    \addplot[only marks, mark options={draw=white, scale = .75}] table [header=false,x index=0,y index=1, col sep=comma] {Genes-10000/RDM-Spearman_005_005_4_YDR191W.csv}; 
    \addplot[mark options={}, dashed] table [header=false,x index=0,y index=1, col sep=comma] {Genes-10000/RDM-Spearman_005_005_4_YDR191W_line.csv};   
\end{axis}
\end{tikzpicture}
\caption{YDR191W}
\end{subfigure}
\begin{subfigure}[b]{0.5\linewidth}
\centering
\begin{tikzpicture}
\begin{axis}[tick style={draw=none}, xticklabels={}, yticklabels={}, width=\linewidth, height=0.63\linewidth]
    \addplot[only marks, mark options={draw=white, scale = .75}] table [header=false,x index=0,y index=1, col sep=comma] {Genes-10000/RDM-Spearman_005_005_5_YLR272C.csv}; 
    \addplot[mark options={}, dashed] table [header=false,x index=0,y index=1, col sep=comma] {Genes-10000/RDM-Spearman_005_005_5_YLR272C_line.csv};   
\end{axis}
\end{tikzpicture}
\caption{YLR272C}
\end{subfigure}\hfill
\begin{subfigure}[b]{0.5\linewidth}
\centering
\begin{tikzpicture}
\begin{axis}[tick style={draw=none}, xticklabels={}, yticklabels={}, width=\linewidth, height=0.63\linewidth]
    \addplot[only marks, mark options={draw=white, scale = .75}] table [header=false,x index=0,y index=1, col sep=comma] {Genes-10000/RDM-Spearman_005_005_6_YNL160W.csv}; 
    \addplot[mark options={}, dashed] table [header=false,x index=0,y index=1, col sep=comma] {Genes-10000/RDM-Spearman_005_005_6_YNL160W_line.csv};   
\end{axis}
\end{tikzpicture}
\caption{YNL160W}
\end{subfigure}
  \caption{\it 
 Transcript levels of the top genes, which were selected  by the  FDR procedure  based on the rearranged  Spearman's rank correlation coefficient, but not by the  FDR procedure  based on  Spearman's rank correlation coefficient.
 ($\alpha = 0.05$).  The dashed lines represent the $3$-nearest neighbour regression estimates.
\label{fig5}
}
\end{figure}
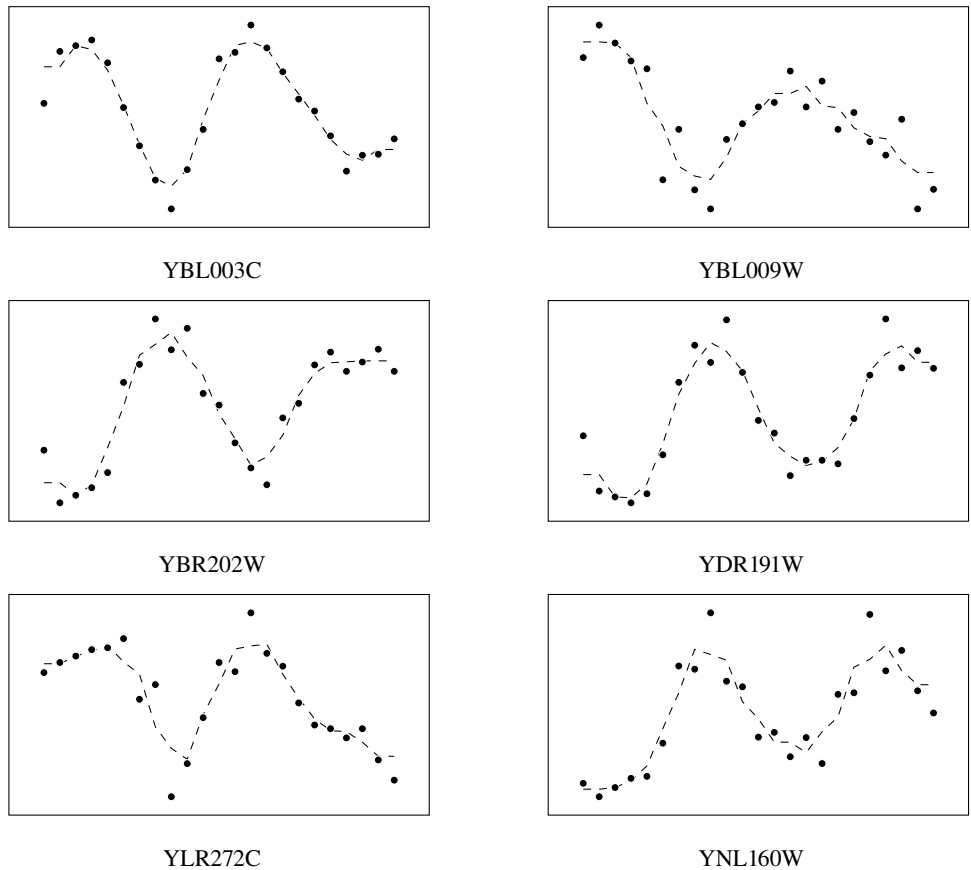

\begin{figure}
\captionsetup[subfigure]{justification=centering, labelformat=empty}
\begin{subfigure}[b]{0.5\linewidth}
\centering
\begin{tikzpicture}
\begin{axis}[tick style={draw=none}, xticklabels={}, yticklabels={}, width=\linewidth, height=0.63\linewidth]
    \addplot[only marks, mark options={draw=white, scale = .75}] table [header=false,x index=0,y index=1, col sep=comma] {Genes-10000/RDM_top_1_YBL003C.csv}; 
    \addplot[mark options={}, dashed] table [header=false,x index=0,y index=1, col sep=comma] {Genes-10000/RDM_top_1_YBL003C_line.csv};   
\end{axis}
\end{tikzpicture}
\caption{YBL003C} 
\end{subfigure}\hfill
\begin{subfigure}[b]{0.5\linewidth}
\centering
\begin{tikzpicture}
\begin{axis}[tick style={draw=none}, xticklabels={}, yticklabels={}, width=\linewidth, height=0.63\linewidth]
    \addplot[only marks, mark options={draw=white, scale = .75}] table [header=false,x index=0,y index=1, col sep=comma] {Genes-10000/RDM_top_2_YBL009W.csv}; 
    \addplot[mark options={}, dashed] table [header=false,x index=0,y index=1, col sep=comma] {Genes-10000/RDM_top_2_YBL009W_line.csv};   
\end{axis}
\end{tikzpicture}
\caption{YBL009W}
\end{subfigure}
\begin{subfigure}[b]{0.5\linewidth}
\centering
\begin{tikzpicture}
\begin{axis}[tick style={draw=none}, xticklabels={}, yticklabels={}, width=\linewidth, height=0.63\linewidth]
    \addplot[only marks, mark options={draw=white, scale = .75}] table [header=false,x index=0,y index=1, col sep=comma] {Genes-10000/RDM_top_3_YBR202W.csv}; 
    \addplot[mark options={}, dashed] table [header=false,x index=0,y index=1, col sep=comma] {Genes-10000/RDM_top_3_YBR202W_line.csv};   
\end{axis}
\end{tikzpicture}
\caption{YBR202W}
\end{subfigure}\hfill
\begin{subfigure}[b]{0.5\linewidth}
\centering
\begin{tikzpicture}
\begin{axis}[tick style={draw=none}, xticklabels={}, yticklabels={}, width=\linewidth, height=0.63\linewidth]
    \addplot[only marks, mark options={draw=white, scale = .75}] table [header=false,x index=0,y index=1, col sep=comma] {Genes-10000/RDM_top_4_YCL040W.csv}; 
    \addplot[mark options={}, dashed] table [header=false,x index=0,y index=1, col sep=comma] {Genes-10000/RDM_top_4_YCL040W_line.csv};   
\end{axis}
\end{tikzpicture}
\caption{YCL040W}
\end{subfigure}
\begin{subfigure}[b]{0.5\linewidth}
\centering
\begin{tikzpicture}
\begin{axis}[tick style={draw=none}, xticklabels={}, yticklabels={}, width=\linewidth, height=0.63\linewidth]
    \addplot[only marks, mark options={draw=white, scale = .75}] table [header=false,x index=0,y index=1, col sep=comma] {Genes-10000/RDM_top_5_YCR098C.csv}; 
    \addplot[mark options={}, dashed] table [header=false,x index=0,y index=1, col sep=comma] {Genes-10000/RDM_top_5_YCR098C_line.csv};   
\end{axis}
\end{tikzpicture}
\caption{YCR098C}
\end{subfigure}\hfill
\begin{subfigure}[b]{0.5\linewidth}
\centering
\begin{tikzpicture}
\begin{axis}[tick style={draw=none}, xticklabels={}, yticklabels={}, width=\linewidth, height=0.63\linewidth]
    \addplot[only marks, mark options={draw=white, scale = .75}] table [header=false,x index=0,y index=1, col sep=comma] {Genes-10000/RDM_top_6_YDL126C.csv}; 
    \addplot[mark options={}, dashed] table [header=false,x index=0,y index=1, col sep=comma] {Genes-10000/RDM_top_6_YDL126C_line.csv};   
\end{axis}
\end{tikzpicture}
\caption{YDL126C}
\end{subfigure}
  \caption{\it Transcript levels of genes, which were selected by the rearranged Spearman's rank correlation coefficient. 
  The figure shows the $6$ top genes with the smallest $p$-values.
   The dashed lines represent the $3$-nearest neighbour regression estimates.
   \label{fig3a} }
\end{figure}
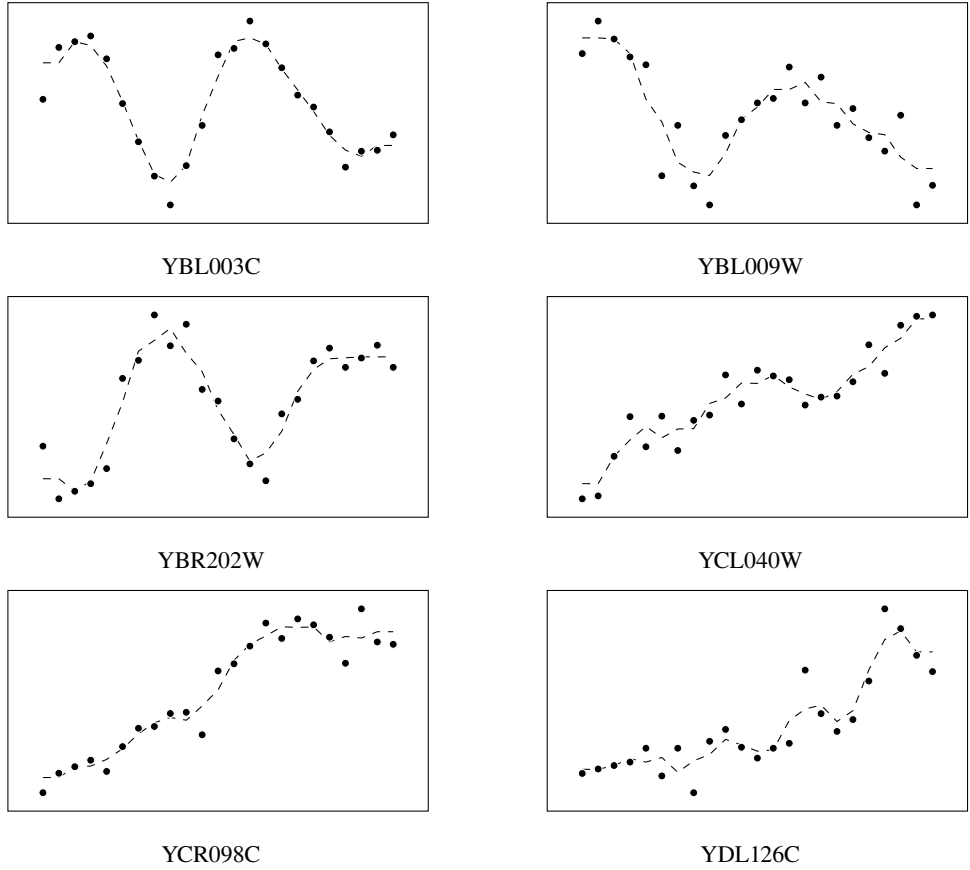
To be precise, we consider the curated data set (available through the R-package \enquote{minerva}) of $4381$ genes. 
For each gene, we  perform a permutation test based on Spearman's rank correlation for the hypotheses 
$$
H_0: 
\rho =0  ~~\text{versus}  ~~H_1: \rho >0 ~
$$
and a permutation test based on the statistic $\hat R_\rho$ for the hypotheses
\begin{equation}
H_0: R_\rho =0  ~~\text{versus}  ~~H_1: R_\rho >0 ~,
\label{deta}
\end{equation}
where we use $10000$ permutations.
The  corresponding $p$-values are use to identify the significant genes 
using the Benjamini–Hochberg FDR procedure with a false discovery rate of $0.05$
\citep[see,][]{benjamini1995}.
To concentrate  on  non-monotone dependencies, we exclude from those genes selected by the FDR procedure based  on the rearranged Spearman's rank correlation all genes which are also detected by Spearman's rank correlation. 
This results in $84$ remaining genes.
In Figure \ref{fig5} we display the  transcript levels of the top $6$ genes 
with the smallest $p$-values from the remaining data.
We observe that the FDR procedure  based on the rearranged Spearman’s rank correlation
identifies additional dependencies, which are oscillating and  are not found if the analysis is   based on  Spearman's rank correlation. A similar
observation was made by \cite{Chatterjee.2020} for his rank correlation coefficient, who used  
$4$ alternative tests to exclude genes with a monotone behaviour (a gene was excluded, whenever one of these tests identified it as significant).
Because both procedures are based on  different dependence measures the 
finally identified $6$ top genes do not necessarily coincide (only the gene YBL003C was selected by our and Chatterjee's procedure). However, all $6$ top genes found by \cite{Chatterjee.2020} are  also selected
by the FDR procedure based on rearranged Spearman’s rank correlation and vice versa.
Moreover,  the qualitative conclusion  from both methods is same. Both methods are able to identify  non-monotone  (in the concrete example oscillating) associations.

\begin{figure}
\captionsetup[subfigure]{justification=centering, labelformat=empty}
\begin{subfigure}[b]{0.5\linewidth}
\centering
\begin{tikzpicture}
\begin{axis}[tick style={draw=none}, xticklabels={}, yticklabels={}, width=\linewidth, height=0.63\linewidth]
    \addplot[only marks, mark options={draw=white, scale = .75}] table [header=false,x index=0,y index=1, col sep=comma] {Genes-10000/Chatterjee_top_1_YJL034W.csv}; 
    \addplot[mark options={}, dashed] table [header=false,x index=0,y index=1, col sep=comma] {Genes-10000/Chatterjee_top_1_YJL034W_line.csv};   
\end{axis}
\end{tikzpicture}
\caption{YJL034W} 
\end{subfigure}\hfill
\begin{subfigure}[b]{0.5\linewidth}
\centering
\begin{tikzpicture}
\begin{axis}[tick style={draw=none}, xticklabels={}, yticklabels={}, width=\linewidth, height=0.63\linewidth]
    \addplot[only marks, mark options={draw=white, scale = .75}] table [header=false,x index=0,y index=1, col sep=comma] {Genes-10000/Chatterjee_top_2_YJR004C.csv}; 
    \addplot[mark options={}, dashed] table [header=false,x index=0,y index=1, col sep=comma] {Genes-10000/Chatterjee_top_2_YJR004C_line.csv};   
\end{axis}
\end{tikzpicture}
\caption{YJR004C}
\end{subfigure}
\begin{subfigure}[b]{0.5\linewidth}
\centering
\begin{tikzpicture}
\begin{axis}[tick style={draw=none}, xticklabels={}, yticklabels={}, width=\linewidth, height=0.63\linewidth]
    \addplot[only marks, mark options={draw=white, scale = .75}] table [header=false,x index=0,y index=1, col sep=comma] {Genes-10000/Chatterjee_top_3_YNL007C.csv}; 
    \addplot[mark options={}, dashed] table [header=false,x index=0,y index=1, col sep=comma] {Genes-10000/Chatterjee_top_3_YNL007C_line.csv};   
\end{axis}
\end{tikzpicture}
\caption{YNL007C}
\end{subfigure}\hfill
\begin{subfigure}[b]{0.5\linewidth}
\centering
\begin{tikzpicture}
\begin{axis}[tick style={draw=none}, xticklabels={}, yticklabels={}, width=\linewidth, height=0.63\linewidth]
    \addplot[only marks, mark options={draw=white, scale = .75}] table [header=false,x index=0,y index=1, col sep=comma] {Genes-10000/Chatterjee_top_4_YKL177W.csv}; 
    \addplot[mark options={}, dashed] table [header=false,x index=0,y index=1, col sep=comma] {Genes-10000/Chatterjee_top_4_YKL177W_line.csv};   
\end{axis}
\end{tikzpicture}
\caption{YKL177W}
\end{subfigure}
\begin{subfigure}[b]{0.5\linewidth}
\centering
\begin{tikzpicture}
\begin{axis}[tick style={draw=none}, xticklabels={}, yticklabels={}, width=\linewidth, height=0.63\linewidth]
    \addplot[only marks, mark options={draw=white, scale = .75}] table [header=false,x index=0,y index=1, col sep=comma] {Genes-10000/Chatterjee_top_5_YDR112W.csv}; 
    \addplot[mark options={}, dashed] table [header=false,x index=0,y index=1, col sep=comma] {Genes-10000/Chatterjee_top_5_YDR112W_line.csv};   
\end{axis}
\end{tikzpicture}
\caption{YDR112W}
\end{subfigure}\hfill
\begin{subfigure}[b]{0.5\linewidth}
\centering
\begin{tikzpicture}
\begin{axis}[tick style={draw=none}, xticklabels={}, yticklabels={}, width=\linewidth, height=0.63\linewidth]
    \addplot[only marks, mark options={draw=white, scale = .75}] table [header=false,x index=0,y index=1, col sep=comma] {Genes-10000/Chatterjee_top_6_YGL089C.csv}; 
    \addplot[mark options={}, dashed] table [header=false,x index=0,y index=1, col sep=comma] {Genes-10000/Chatterjee_top_6_YGL089C_line.csv};   
\end{axis}
\end{tikzpicture}
\caption{YGL089C}
\end{subfigure}
  \caption{\it Transcript levels of genes, which were selected by the Chatterjee's correlation coefficient. 
  The figure shows the $6$ genes with the smallest $p$-values.
  The dashed lines represent the $3$-nearest neighbour regression estimates.
  \label{fig3b}
 }
\end{figure}
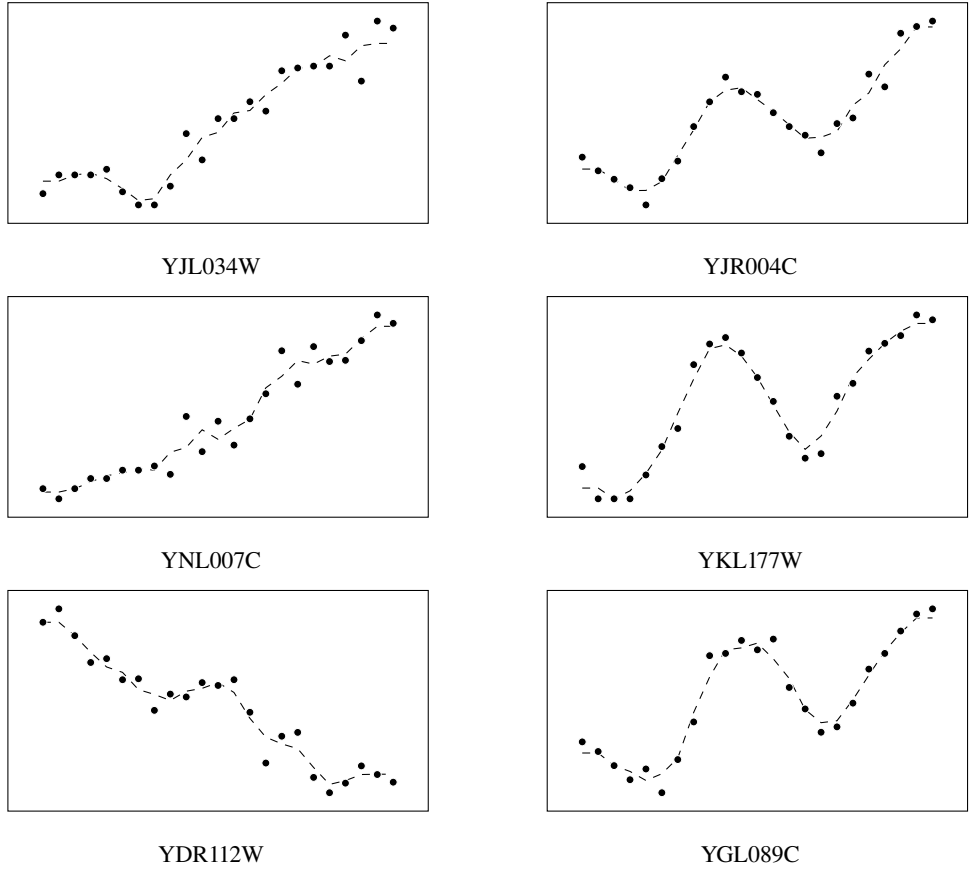

We conclude with a brief comparison of the FDR procedures based on the rearranged Spearman's 
and Chatterjee's rank correlation coefficient, if they are used without sorting out monotone dependencies
by preliminary analysis. In Figures \ref{fig3a}  and \ref{fig3b}, we display the 
 transcript levels of the $6$  genes with the smallest $p$-values after running the FDR procedure 
based on the two dependency measures. We observe again that both methods are able to 
identify non-monotone associations. Interestingly the top 
three genes identified by the rearranged Spearman’s rank correlation with the smallest three $p$-values exhibit an oscillating transcript level while
it looks more monotone for the next three genes. For the FDR procedure based on Chatterjee's rank correlation 
the picture is not so clear.

\section{Conclusions and outlook}
\HD{
In this paper we  developed  a general strategy to transform a given dependence measure into a new one which exactly characterizes independence (with the value $0$)  as well as functional  dependence (with the value $1$).  The approach is applicable 
to many of the commonly used dependence measures and we have
also developed   consistent estimates of the new dependence measures. 
An interesting question of future research is the asymptotic distribution of the new estimators.  However, such an investigation  will be very challenging 
because -  in contrast to  most of the literature  - we do not consider  “one specific” estimator for “one specific” dependence measure, for which the asymptotic distribution is established.
Thus one would have to identify classes of copulas and classes dependence measures for which such an asymptotic analysis is possible. 
\\
A further challenging question are  rates of convergence of the new estimators. These rates  are case-specific and will depend in an intricate way on the quality of the approximation of the copula $C$ by the induced checkerboard copula $C^{\#}_{N, N, n}(C) $  (which might be different for different copulas) and on the dependence measure under consideration.
To indicate how these  rates can be obtained in principle, we consider exemplary the case of Spearman's $\rho$ and note that 
\begin{align*}
    \left| R_\rho(C) - R_\rho(D) \right|    &= 12 \Big | \int _{[0, 1]^2} \ir{C}(u, v) - \ir{D}(u, v) \; \mathrm{d}\lambda(u, v) \Big| \\
                                            &= 12 \Big | \int_{[0, 1]^2} \int_{0}^u \partial_1 \ir{C}(s, v) - \partial_1 \ir{D}(s, v) \; \mathrm{d} s \; \mathrm{d}\lambda(u, v) \Big | \\
                                            &\leq  12 \int_{[0, 1]^2} \int_{0}^u \left|\partial_1 \ir{C}(s, v) - \partial_1 \ir{D}(s, v) \right| \; \mathrm{d} s \; \mathrm{d}\lambda(u, v)  \\   
                                            &\leq 12 D_1(\ir{C}, \ir{D})
                                            \leq 12  D_1(C, D) ~.
\end{align*}
Thus, in order to obtain the rate for the rearranged estimate of Spearman's $\rho$ we need  an estimate for  $  D_1(\hat{C}^{\#}_{N, N, n} ,C)$. 
 For this purpose we can  use results of  
\cite{Junker.2021}, who showed the inequality 
\begin{equation*}
    D_1(\hat{C}^{\#}_{N, N, n} ,C)    \leq K 
    \frac{\sqrt{\log \log n}}{n^{{1}/{2} - s}} + D_1(C^{\#}_{N, N, n}(C),C)
\end{equation*}
(almost surely), and it remains to estimate
the  deterministic quantity $D_1(C^{\#}_{N, N, n}(C),C)$.
Estimates of this type are again case specific. For example,
if $C = \Pi $ is the independence copula it follows that $C^{\#}_{N, N, n}(\Pi) = \Pi$ and $D_1(C^{\#}_{N, N, n}(\Pi), \Pi) = 0$.
Thus, we obtain in this case  
$$ |\hat R_\rho  - R_\rho (\Pi)  | = O \Big ( \frac{\sqrt{\log \log n}}{n^{{1}/{2} - s}} \Big )  
$$
(almost surely) for $0 < s < \frac{1}{2}$.
Note that the distance $D_1$ appears here, because of the representation $\rho(C)  = 12 \cInt{C(u, v)}{[0, 1]^2}{}{\lambda(u, v)} - 3 $ (see, e.g., Chapter~5 in \cite{Nelsen.2006}). A similar  argument can be used to derive a rate for 
the estimator of the rearranged  Blum-Kiefer-Rosenblatt's $R$.
Other dependence measures such as Kendall's $\tau$ or Gini's $\gamma$ have more complicated representations in terms of the copula $C$,  which might require other metrics to bound  $ \left| R_\mu (C) - R_\mu (D) \right| $. 
}
\appendix

\section{Proofs of the results in Section \ref{sec2}} \label{appendix:section2}

\customAppendixHeading{Proof of Theorem~\ref{thm:rearrangement}}
In order to show that the stochastically increasing rearrangement, $\ir{C}$ is a copula, we verify the properties (1) to (3) of Definition~\refdefcopula{}.
\begin{enumerate}
	\item It follows from $(\partial_1 C)^*(u, 0) = 0^* = 0$ that $\ir{C}(u, 0) = 0$. The identity $\ir{C}(0, v) = 0$ is trivial by Definition~\ref{def2.2}. 
	\item By definition, we have 
	\begin{align*}	
		\ir{C}(u, 1)	= \cInt{(\partial_1 C)^*(s, 1)}{0}{u}{s}
						= \cInt{1^*}{0}{u}{s}
						= u ~.
	\end{align*}
	In view of Proposition~\refpropertiesdecreasingrearrangement(3), we further obtain  that
	\begin{equation*}
		\ir{C}(1, v)	= \cInt{(\partial_1 C)^*(s, v)}{0}{1}{s}
						= \cInt{(\partial_1 C)^*(\sigma_v(s), v)}{0}{1}{s}
						= \cInt{\partial_1 C(t, v)}{0}{1}{t}
						= v ~.
	\end{equation*}
	\item From Definition~\refdefcopula{}(3) we see that $0 \leq \partial_1 C(\cdot, v_1) \leq \partial_1 C(\cdot, v_2)$ whenever $v_1\leq v_2$. 
	Combining this with Proposition~\refpropertiesdecreasingrearrangement{}(2) yields $(\partial_1 C)^*(\cdot, v_1) \leq (\partial_1 C)^*(\cdot, v_2)$. 
	Thus, the $\ir{C}$-volume of a rectangle $[u_1, u_2) \times [v_1, v_2)$ satisfies
	\begin{align*}
		V_{\ir{C}} \rbraces{[u_1, u_2) \times [v_1, v_2)}	&= \ir{C}(u_2, v_2) - \ir{C}(u_1, v_2) - \ir{C}(u_2, v_1) + \ir{C}(u_1, v_1) \\
															&= \cInt{(\partial_1 C)^*(s, v_2) - (\partial_1 C)^*(s, v_1)}{u_1}{u_2}{s} \geq 0 .
	\end{align*}
\end{enumerate}

Finally, we show that $C$ is stochastically increasing if and only if $C = \ir{C}$. If $C = \ir{C}$, of course, $C$ is stochastically increasing because $\ir{C}$ is. Conversely, suppose $C$ is stochastically increasing, i.e., each $u\mapsto C(u, v)$ is concave.
Then the right-hand derivative $u\mapsto \partial_1^+ C(u, v)$ is a decreasing and right-continuous function, and \cite[Thm.~4.2]{Chong.1971} guarantees that $\partial_1^+ C(u, v) = (\partial_1 C)^*(u, v)$. This implies
\begin{equation*}
	C(u, v)	= \cInt{\partial_1^+ C(t, v)}{0}{u}{t}
			= \cInt{(\partial_1 C)^*(t, v)}{0}{u}{t}
			= \ir{C}(u, v) ~.
\end{equation*}

\customAppendixHeading{Proof of Theorem \ref{thm:good_measure_R}} 
We will require a preliminary result. For this, we first note that the so-called (SD)-rearrangement of $C$ defined by
\begin{equation*}
	\dr{C}	(u,v)   := \cInt{(\partial_1 C)^*(1-s, v)}{0}{u}{s}
			        = v - \ir{C}(1-u, v) 
			        = (C^-*\ir{C})(u,v)
\end{equation*}
is a stochastically decreasing copula.

\begin{lemma} \label{lemma:hardy_littlewood_polya_copula}
For any copula $C$, we have $\dr{C} (u, v)   \leq C(u, v)    \leq \ir{C}(u, v)$.
\end{lemma}

\begin{proof}
By Theorem~\refproporderproperties{}(a) we obtain the upper estimate
\begin{equation*}
	C (u, v)	= \cInt{\indFunc{[0, u]}(t) \partial_1 C(t, v)}{0}{1}{t} 
				\leq \cInt{\indFunc{[0, u]}(t)  (\partial_1 C)^*(t, v)}{0}{1}{t} 
				= \ir{C}(u, v) ~.
\end{equation*}
The lower estimate follows analogously.
\end{proof}

We will now prove properties (1.1)--(1.3) for $R_{\mu}(C) = \mu(\ir{C})$.
For this, we say that the copula $C$ is \emph{completely dependent} if there exists a measurable function $f$ such that $V = f(U)$.
It is proven in \cite{Darsow.1992} that $C$ is completely dependent if, and only if,
\begin{equation} \label{eq:compldep}
\partial_1 C(u, v) \in \curly{0, 1}
\end{equation}
for almost all $u \in [0, 1]$ and all $v \in [0, 1]$.

\begin{enumerate}
\item[(1.1)] Since $\mu$ only takes values between $0$ and $1$, we obtain the first assertion.
\item[(1.2)] If $C = \Pi$, we have $\mu(\ir{C}) = \mu(\ir{\Pi}) = \mu(\Pi) = 0$. 
If, on the other hand,  $\mu(\ir{C}) = 0$, we conclude $\ir{C} = \Pi$ by the properties of $\mu$. But then $\dr{C} = C^-*\Pi = \Pi$, and Lemma~\ref{lemma:hardy_littlewood_polya_copula} yields $\Pi = \dr{C} \leq C \leq \ir{C} = \Pi$, hence $C=\Pi$.
\item[(1.3)] If $C$ is completely dependent, then $\ir{C} = C^+$ and $\mu(\ir{C}) = \mu(C^+) = 1$ by definition.
On the other hand, $\mu(\ir{C}) = 1$ implies $\ir{C} = C^+$ by the properties of $\mu$.
Thus, $\partial_1 C(u, v) = (\partial_1 C)^*(\sigma_v(u), v) \in \curly{0, 1}$, so $C$ is completely dependent by \eqref{eq:compldep}.
\end{enumerate}

\customAppendixHeading{Proof of Equation \eqref{det2}}
The statement is an immediate consequence of the fact that the decreasing rearrangement of $g_v(u) := \partial_1 C(u, v) - v$ is $g^*_v(u) = \partial_1 \ir{C}(u, v) - v$.
As the decreasing rearrangement leaves all $L^p$-norms invariant, we conclude
\begin{equation*}
	\cInt{\abs{\partial_1 C(u, v) - v}^p}{0}{1}{u}	= \cInt{\abs{\partial_1 \ir{C}(u, v) - v}^p}{0}{1}{u} ~.
\end{equation*}
Integrating with respect to $v$ yields the desired result with $p=1,2$.

\customAppendixHeading{Proof of the statements in  Example \ref{prop:measure_lp}} 
In this section we show that the Schweizer-Wolff measure $\sigma_p$ in \eqref{det11} for $1 \leq p < \infty$ satisfies the properties (1.1) to (1.3) on the set $\ir{\C}$.

\begin{enumerate}
    \item[(1.1)] $\sigma_p$ takes values only between $0$ and $1$, since $\ir{C}$ is stochastically increasing and fulfils $0  \leq \ir{C}  -  \Pi \leq  C^+ -  \Pi$.
	\item[(1.2)] $\sigma_p(C) = 0$ holds if and only if $C= \Pi$. 
	\item[(1.3)] Suppose $C = \ir{C}$ is completely dependent. Then $\partial_1 \ir{C}(u, v) \in \curly{0, 1}$ by \eqref{eq:compldep} and $\partial_1 \ir{C}(u, v) = \indFunc{[0, v]}(u)$ by Definition~\refdefcopula(2).
	Thus, $\ir{C} = C^+$ which yields $\sigma_p(C) = 1$. 
	On the other hand, if $C$ is not completely dependent, then an analogous argument shows that $\ir{C} < C^+$ on a set of positive measure such that
	\begin{equation*}
		\sigma_p(C)	= \frac{\norm{\ir{C} - \Pi}_p}{\norm{C^+ - \Pi}_p}	< \frac{\norm{C^+ - \Pi}_p}{\norm{C^+ - \Pi}_p} = 1 ~.
	\end{equation*}
\end{enumerate}

\customAppendixHeading{Proof of the statements in  Example \ref{ex:kappa_measures}} 
We introduce the concordance functional
\begin{equation*}
    Q(C_1, C_2) := 4 \cInt{C_1(u, v)}{[0, 1]^2}{}{C_2(u, v)} - 1~
\end{equation*}
and  point out for later reference that $Q$ is symmetric and fulfils
\begin{equation} \label{eq:Qmonotone}
Q(C_1, C_2) \leq Q(C_1', C_2)
\end{equation}
whenever $C_1 \leq C_1'$.
Then  the four measures of concordance (see Definition~\refdefconcordance{}) Spearman's $\rho$, Kendall's $\tau$, Gini's $\gamma$ and Blomqvist's $\beta$ are given by (see, e.g., Chapter~5 in \cite{Nelsen.2006})
\begin{align*}
\rho(C) &= 3 Q(C,\Pi) = 12 \cInt{C(u, v)}{[0, 1]^2}{}{\lambda(u, v)} - 3 \\
\tau(C) &= Q(C,C) = 4 \cInt{C(u, v)}{[0, 1]^2}{}{C(u, v)} - 1 \\
\gamma(C) &= Q(C,C^-) + Q(C,C^+) = 2 \cInt{|u+v-1|-|u-v|}{[0, 1]^2}{}{C(u, v)} \\
\beta(C) &= 4 C\rbraces{\frac{1}{2},\frac{1}{2}} - 1 ~.
\end{align*}

First of all, $\beta$ does not satisfy (1.3) on $\ir{\C}$ because
the copula\footnote{$C$ is a so-called ordinal sum; see \cite[Sect.~3.2.2]{Nelsen.2006}.}
\begin{equation*}
	C(u, v)	= \begin{cases}
					2\Pi(u, v)		&\text{ if } (u, v) \in [0, 1/2]^2 \\
					C^+(u, v)		&\text{ else }
				\end{cases}
\end{equation*}
is stochastically increasing with $C\neq C^+$, yet $\beta(C) = 4 C(1/2, 1/2) - 1 = 1 = \beta(C^+)$. 

We now show that $\rho, \tau$ and $\gamma$ all satisfy the properties (1.1)--(1.3) on $\ir{\C}$. 
Since any concave function $f: [0,1]\to [0,v]$ with $f(0)=0$ and $f(1)=v$ satisfies $f(u)\geq uv= \Pi(u,v)$, any stochastically increasing copula $C$ satisfies
\begin{equation} \label{eq:posquad}
\Pi \leq C = \ir{C} \leq C^+ ~.
\end{equation}
Hence we conclude from Definition~\refdefconcordance (4) that $0 = \kappa(\Pi) \leq \kappa(\ir{C}) \leq \kappa(C^+) = 1$. 
It remains to verify properties (1.2) and (1.3) for $\rho, \tau$ and $\gamma$. 

First, we look at Spearman's $\rho$. By Proposition~\ref{prop:spearman}, $R_\rho$ coincides with $R_{\sigma_1}$ so that, in view of Example \ref{prop:measure_lp} with $p=1$, the properties (1.2) and (1.3) hold.

Next, consider Kendalls's $\tau$. In order to prove (1.2), we assume $\tau(C) = \tau(\Pi)$, i.e.\ $Q(C,C) = Q(\Pi,\Pi)$, for some $C \in \ir{\C}$. In view of \eqref{eq:Qmonotone} and \eqref{eq:posquad} we obtain $Q(\Pi, \Pi) \leq Q(C, \Pi)  \leq  Q(C, C) = Q(\Pi,\Pi)$ so that
\begin{align*}
	0 	&\leq 4\cInt{\abs{C(u, v) - \Pi(u, v)}}{[0, 1]^2}{}{\lambda(u, v)} \\
		&= 4\cInt{C(u, v) - \Pi(u, v)}{[0, 1]^2}{}{\lambda(u, v)}
		= Q(C, \Pi) - Q(\Pi, \Pi)
		= 0
\end{align*}
which indeed implies $C = \Pi$. 
For the proof of (1.3), we suppose $\tau(C) = \tau(C^+)$, i.e.\ $Q(C,C) = Q(C^+,C^+)$.
In view of \eqref{eq:Qmonotone} and \eqref{eq:posquad} we obtain $Q(C,C) \leq Q(C,C^+)  \leq  Q(C^+,C^+) = Q(C,C)$ so that
\begin{align*}
	0 	&\leq 4\cInt{\abs{u - C(u, u)}}{0}{1}{u}
	    = 4\cInt{u - C(u, u)}{0}{1}{u} \\
		&= 4 \cInt{C^+(u, v) - C(u, v)}{[0, 1]^2}{}{C^+(u, v)}
		= Q(C^+, C^+) - Q(C, C^+) 
		= 0 ~.
\end{align*}
Therefore $C(u, u) = u$ for all $u \in [0, 1]$ so that $C = C^+$ (see \cite[Ex~2.6.4]{Durante.2015}).\footnote{The observation that $\tau(C) = \tau(C^+)$ implies $C=C^+$ also in the multivariate case is contained in \cite[Thm.~3.2]{Fuchs.2018}.}

Finally, we turn to Gini's $\gamma$.
In order to prove (1.2), we assume $\gamma(C) = \gamma(\Pi)$, i.e.\ $$Q(C, C^+) + Q(C, C^-) = Q(\Pi, C^+) + Q(\Pi, C^-) ~, $$ for some $C \in \ir{\C}$. 
In view of \eqref{eq:Qmonotone} and \eqref{eq:posquad} we obtain 
\begin{align*}
    Q(\Pi, C^+) + Q(\Pi, C^-)   \leq Q(C, C^+) + Q(\Pi, C^-)  
                                \leq  Q(C, C^+) + Q(C, C^-)  
                                = Q(\Pi, C^+) + Q(\Pi, C^-)
\end{align*}
so that
\begin{equation*}
	0 	\leq 4\cInt{\abs{C(u, u) - \Pi(u, u)}}{0}{1}{u}
		= 4\cInt{C(u, u) - \Pi(u, u)}{0}{1}{u}
		= Q(C, C^+) - Q(\Pi, C^+)
		= 0 ~.
\end{equation*}
It follows that $C(u, u) = \Pi(u,u)$, and Proposition~2.1 in \cite{Durante.2009} yields $C=\Pi$.
For the proof of (1.3), we suppose $\gamma(C) = \gamma(C^+)$, i.e.\ 
$$Q(C, C^+) + Q(C, C^-) = Q(C^+, C^+) + Q(C^+, C^-). 
$$
In view of \eqref{eq:Qmonotone} and \eqref{eq:posquad} we obtain 
\begin{align*}
Q(C, C^+) + Q(C, C^-)   \leq Q(C^+, C^+) + Q(C, C^-) 
                        \leq Q(C^+, C^+) + Q(C^+, C^-)  
                        = Q(C, C^+) + Q(C, C^-) ~,
\end{align*}
which implies
\begin{align*}
	0 	&\leq 4\cInt{\abs{u - C(u, u)}}{0}{1}{u}
	    = 4\cInt{u - C(u, u)}{0}{1}{u} \\
		&= 4 \cInt{C^+(u, v) - C(u, v)}{[0,1]^2}{}{C^+(u, v)} 
		= Q(C^+, C^+) - Q(C, C^+) 
		= 0 ~.
\end{align*}
Therefore $C(u, u) = u$ for all $u \in [0, 1]$ so that $C = C^+$ \cite[Ex.~2.6.4]{Durante.2015}.

\customAppendixHeading{Proof of Proposition \ref{prop:spearman}}
This follows readily from the fact that $\ir{C} \geq \Pi$ since 
\begin{align*}
R_{\sigma_1}(C) & = \frac{\norm{\ir{C} - \Pi}_1}{\norm{C^+ - \Pi}_1} = 12 \cInt{\ir{C}(u, v)-uv}{[0, 1]^2}{}{\lambda(u, v)} \\
& = 12 \cInt{\ir{C}(u, v)}{[0, 1]^2}{}{\lambda(u, v)} - 3 = R_{\rho}(C) ~. 
\end{align*}

\customAppendixHeading{Proof of Theorem \ref{cor:spearman_positive}}
In view of Definition~\refdefconcordance{}, we have $\kappa(\dr{C}) = \kappa(C^-*\ir{C}) = -\kappa(\ir{C})$. Consequently, we know from Lemma~\ref{lemma:hardy_littlewood_polya_copula} and the monotonicity of $\kappa$ with respect to the pointwise ordering that
\begin{equation*}
	- \kappa(\ir{C}) = \kappa(\dr{C}) \leq \kappa(C) \leq \kappa(\ir{C})~,
\end{equation*}
which implies $\abs{\kappa(C)} \leq \kappa (\ir{C}) = \RDk(C)$. Moreover, if $C$ is stochastically monotone we have $C=\dr{C}$ or $C=\ir{C}$ and, therefore, $\abs{\kappa(C)} = \kappa (\ir{C})$.

\customAppendixHeading{Proof of Proposition \ref{prop:dpi}}
First, we point out that that the Markov product of two copulas $C$ and $D$ satisfies
\begin{equation} \label{eq:majorization_order}
    \partial_1 (C * D)(\cdot ,v) = \partial_u \cInt{\partial_2 C(\cdot, t) \cdot \partial_1 D(t, v)}{0}{1}{t} \preceq \partial_1 D(\cdot, v)
\end{equation}
for all $v \in [0, 1]$, where ``$\preceq$'' denotes   the majorization order 
introduced in Definition~\refdefmajorization{}. 
This follows from Theorem~\refproporderproperties{}(3) and the fact that $\partial_1 (C*D)(u, v) = T_C \partial_1 D (\cdot, v) (u)$. 
In particular,
\begin{equation*}
    \ir{(C * D)}(u, v) \leq \ir{D} (u, v) ~. 
\end{equation*}
Now suppose $X, Y$ and $Z$ are continuous random variables such that $Y$ and $Z$ are conditionally independent given $X$.
Then $C_{ZY} = C_{ZX} * C_{XY}$ in view of  Theorem~3.1 in \cite{Darsow.1992}, and \eqref{eq:majorization_order} yields
\begin{equation*}
	\ir{C}_{ZY} = \ir{(C_{ZX} * C_{XY})}	\leq \ir{C}_{XY} ~.
\end{equation*}
Thus, the data processing inequality $\RDm(C_{ZY}) = \mu(\ir{C}_{ZY}) \leq \mu(\ir{C}_{XY}) = \RDm(C_{XY})$ follows from the monotonicity of $\mu$.

\customAppendixHeading{Proof of Corollary~\ref{cor:self_equitability}}
The data processing inequality in Proposition~\ref{prop:dpi} states that 
$$R_{\mu}(f(X),Y) \leq R_{\mu}(X,Y)$$
for all measurable functions $f$. If, in addition, $X$ and $Y$ are independent given $f(X)$, a second application of Proposition~\ref{prop:dpi} yields $R_{\mu}(X,Y) \leq R_{\mu}(f(X),Y)$, and equality holds.

\section{Proofs of the results in Section \ref{sec3} }
\label{appendix:section3}

\customAppendixHeading{Proof of  Theorem \ref{thm:CB_algorithm}}
The equality $\ir{\CBnn{A}} = \CBnn{\ir{A}}$ follows directly from the definition of  Algorithm
\ref{alg1} and the characterization \eqref{det5}. 
It remains to show that  the matrix $\ir{A}$ satisfies  indeed the properties in \eqref{det20}.
To do so, we calculate 
\begin{equation*}
	\sum\limits_{\ell=1}^{N_2} \ir{a}_{k\ell}	= \sum\limits_{\ell =1}^{N_2} \widetilde{B}^\ell_k - \widetilde{B}^{\ell-1}_k 
												= \widetilde{B}^{N_2}_k - \widetilde{B}^0_k 
												= \widetilde{B}^{N_2}_k 
												= \sum\limits_{\ell = 1}^{N_2} a_{k\ell}
												= N_2 
\end{equation*}
as well as
\begin{align*}
	\sum\limits_{k = 1}^{N_1} a_{k \ell}&	= \sum\limits_{k = 1}^{N_1} \widetilde{B}^\ell_k - \widetilde{B}^{\ell-1}_k 
											= \sum\limits_{k = 1}^{N_1} {B}^\ell_k - {B}^{\ell-1}_k  
											= \sum\limits_{j = 1}^\ell \sum\limits_{k =1}^{N_1} a_{k j} - \sum\limits_{j = 1}^{\ell-1} \sum\limits_{k =1}^{N_1} a_{k j}
											= \ell N_1 - (\ell - 1) N_1 = N_1 ~. \qedhere
\end{align*}
The nonnegativity of $\ir{a}_{k \ell}$ follows by construction.

\customAppendixHeading{Proof of Theorem \ref{thm:consistency}}
We will start by showing a contraction property of the (SI)-rearrangement with respect to $D_p$.
For all copulas $C$ and $D$, it holds by Theorem~\refproporderproperties{}(b)
\begin{equation*}
	\partial_1 \ir{C}(\cdot, v) - \partial_1  \ir{D}(\cdot, v) \preceq \partial_1 C(\cdot, v) - \partial_1 D(\cdot,v)
\end{equation*}
for all $v$ in $[0, 1]$, where ``$\preceq$'' denotes the majorization order introduced in Definition~\refdefmajorization{}.
Thus, due to Theorem~\refproporderproperties{}, we have for all $v \in [0, 1]$ and any $1 \leq p < \infty$
\begin{equation*}
	\cInt{ \abs{\partial_1 \ir{C}(u, v) - \partial_1  \ir{D}(u, v)}^p}{0}{1}{u} \leq \cInt{\abs{\partial_1 C(u, v) - \partial_1 D(u, v)}^p}{0}{1}{u} ~.
\end{equation*}
and integrating with respect to $v$ yields $D_p(\ir{C}, \ir{D}) \leq D_p(C, D)$.
Now it follows  by similar arguments as in the proof of Theorem~4.5.8 in \cite{Durante.2015}
(these authors considered the case $N_1 = N_2$)
that 
\begin{equation*}
	0\leq D_p(\ir{\CBnn{C}}, \ir{C})	\leq D_p (\CBnn{C}, C)  \rightarrow 0 ~. 
\end{equation*}

\customAppendixHeading{Proof of Theorem \ref{thm:estimation}}
The almost sure convergence of $D_1(\hat C^\#_{N_1, N_2, n}, C) \rightarrow 0$ follows from Theorem~3.12 in \cite{Junker.2021}, where $\hat C^\#_{N_1, N_2, n}$ is a genuine copula.
Thus, an application of the continuity property given in Theorem~\ref{thm:consistency} implies
\begin{equation*}
	0	\leq D_1((\hat C^{\#}_{N_1, N_2, n})^{\uparrow}, \ir{C})	\leq   D_1(\hat C^\#_{N_1, N_2, n}, C) \rightarrow 0 ~.
\end{equation*}
and therefore $\hat R_\mu \rightarrow \RDm(C)$ almost surely. 

\begin{supplement}
\stitle{Supplement to \enquote{Rearranged dependence measures}}
\sdescription{This supplement contains  basic facts about copulas and monotone rearrangements and provides proofs of the multivariate results of Section~\ref{section:multivariate}.}
\end{supplement}

\begin{funding}
C. Strothmann gratefully acknowledges financial support from the German Academic Scholarship Foundation.
The work of H. Dette was supported by the DFG Research Unit 5381 {\it Mathematical Statistics in the Information Age}. The authors are grateful to two referees for their constructive comments on an earlier version of this paper.
\end{funding}

\bibliographystyle{imsart-nameyear} 
\bibliography{reference}



\section*{Supplementary material}

# Online supplement to: Rearranged dependence measures

CHRISTOPHER STROTHMANN[1,a], HOLGER DETTE[2,c] and KARL FRIEDRICH SIBURG[1,b]

[1]*Department of Mathematics, TU Dortmund University, Vogelpothsweg 87, 44221 Dortmund, Germany,*
[a]*christopher.strothmann@mathematik.tu-dortmund.de,* [b]*karl.f.siburg@mathematik.tu-dortmund.de*
[2]*Department of Mathematics, Ruhr-University Bochum, Universitätsstraße 150, 44780 Bochum, Germany,*
[c]*holger.dette@rub.de*

This online supplement to Strothmann, Dette and Siburg (2022) contains two parts. S.I collects some basic facts about copulas as well as monotone rearrangements. S.II presents the proofs of the multivariate results of Section 2.4.

## Supplement S.I: Results from the literature

In this section, we present some basic facts about copulas and monotone rearrangements, which will be frequently used throughout the proofs of our results in Appendix A and B. We start with the definition of a bivariate copula, which is a distribution function on the unit square with uniform univariate margins.

**Definition S.I.1.** A function $C : [0,1]^2 \to [0,1]$ is called a (bivariate) copula if

1. $C$ is grounded, i.e. $C(0,v) = C(u,0) = 0$ for all $u,v \in [0,1]$
2. $C$ has uniform margins, i.e. $C(1,u) = C(u,1) = u$ for all $u \in [0,1]$
3. $C$ is 2-increasing, i.e. the $C$-volume of every rectangle $R = [u_1,u_2] \times [v_1,v_2]$ is nonnegative:

$$V_C(R) := C(u_2,v_2) - C(u_1,v_2) - C(u_2,v_1) + C(u_1,v_1) \geq 0 \,.$$

The set of all copulas is denoted by $\mathcal{C}$. We refer to the lower Fréchet-Hoeffding bound by $C^-(u,v) := \max\{u+v-1, 0\}$, to the independence (or product) copula by $\Pi(u,v) := uv$, and to the upper Fréchet-Hoeffding bound by $C^+(u,v) := \min\{u,v\}$. Any copula $C$ satisfies $C^- \leq C \leq C^+$.

**Definition S.I.2.** The Markov product of two copulas $C$ and $D$ is defined as the copula

$$(C * D)(u,v) := \int_0^1 \partial_2 C(u,t) \partial_1 D(t,v) \, \mathrm{d}t \,.$$

A comprehensive review of the Markov product can be found in Durante and Sempi (2016).

**Definition S.I.3.** A linear operator $T : L^1([0,1]) \to L^1([0,1])$ is called a Markov operator if

1. $T$ is positive, i.e. $Tf \geq 0$ whenever $f \geq 0$
2. $T \mathbb{1}_{[0,1]} = \mathbb{1}_{[0,1]}$
3. $T$ preserves the integral, i.e. $\int_0^1 Tf(t) \, \mathrm{d}t = \int_0^1 f(t) \, \mathrm{d}t$ for all $f \in L^1([0,1])$.

The following result shows that copulas and Markov operators are closely linked and that the composition of Markov operators corresponds to the Markov product of copulas. A proof can be found in Olsen, Darsow and Nguyen (1996).

**Theorem S.I.4.** *Let C be a copula and T be a Markov operator. Then*

$$C_T(u,v) := \int_0^u T\mathbb{1}_{[0,v]}(t)\,dt \quad \text{and} \quad T_C f(u) := \partial_u \int_0^1 \partial_2 C(u,v) f(v)\,dv$$

*define a copula $C_T$ and a Markov operator $T_C$, respectively. The correspondence $C \mapsto T_C$ is bijective with $T_{C_T} = T$ and $C_{T_C} = C$. Moreover,*

$$T_{C_1 * C_2} = T_{C_1} \circ T_{C_2}$$

*holds for all copulas $C_1$ and $C_2$.*

The following definition of a concordance measure is adapted from Durante and Sempi (2016).

**Definition S.I.5.** A function $\kappa : \mathcal{C} \to [-1,1]$ is called a measure of concordance if

1. $\kappa(C^-) = -1$, $\kappa(\Pi) = 0$ and $\kappa(C^+) = 1$
2. $\kappa(C^\top) = \kappa(C)$, where $C^\top(u,v) := C(v,u)$
3. $\kappa(C^- * C) = \kappa(C * C^-) = -\kappa(C)$
4. $\kappa$ is monotone w.r.t. the pointwise order on the set of copulas
5. $\kappa$ is continuous w.r.t. the pointwise[1] convergence of copulas.

For the decreasing rearrangement $f^* : [0,1] \to \mathbb{R}$ of a measurable function $f : [0,1]^d \to \mathbb{R}$, we state the following properties.

**Proposition S.I.6.** *For any two measurable functions f and g on the measure space $([0,1]^d, m)$, the following assertions hold:*

1. *$f^*$ is decreasing and right-continuous on $[0,1]$.*
2. *$f \leq g$ implies $f^* \leq g^*$.*
3. *There exists an $m$-$\lambda$-preserving transformation $\sigma : [0,1]^d \to [0,1]$ such that $f = f^* \circ \sigma$.*
4. *The decreasing rearrangement is $L^p$-invariant for $1 \leq p \leq \infty$, i.e.*

$$\|f\|_p = \|f^*\|_p .$$

**Proof.** Property (1) is stated in Theorem 4.2, properties (2) and (4) can be found in Proposition 4.3, and property (3) is stated in Theorem 6.2 of Chong and Rice (1971). □

Closely linked to the decreasing rearrangement of measurable functions is an ordering widely known as the majorization order, introduced by Hardy, Littlewood and Pólya for vectors, and by Ryff (1965) for functions.

---

[1] As copulas are continuous function on a compact set, pointwise and uniform convergence are equivalent.

**Definition S.I.7.** Suppose $f, g \in L^1([0,1], \lambda)$. Then $f$ is majorized by $g$, denoted by $f \preceq g$, if

$$\int_0^t f^*(s) \, ds \leq \int_0^t g^*(s) \, ds$$

holds for all $t \in [0, 1]$, as well as

$$\int_0^1 f^*(s) \, ds = \int_0^1 g^*(s) \, ds .$$

**Theorem S.I.8.** *For $f, g \in L^1([0,1], \lambda)$, the following statements are equivalent:*

1. *$f$ is majorized by $g$, i.e. $f \preceq g$.*
2. *For every convex function $\phi : \mathbb{R} \to \mathbb{R}$ we have*

$$\int_0^1 \phi(f(s)) \, ds \leq \int_0^1 \phi(g(s)) \, ds .$$

3. *There exists a Markov operator $T$ such that $f = Tg$.*

*Furthermore, the following inequalities hold:*

(a)

$$\int_0^1 |f^*(s) g^*(1-s)| \, ds \leq \int_0^1 |f(s) g(s)| \, ds \leq \int_0^1 |f^*(s) g^*(s)| \, ds .$$

(b)

$$f^* - g^* \preceq f - g .$$

**Proof.** The equivalence of (1) and (3) is shown in (Day, 1973, Thm. 4.9), while that of (1) and (2) is contained in (Chong, 1974, Thm. 2.5). The proofs of (a), called the Hardy-Littlewood inequality, and (b) can be found in (Day, 1972, (6.2) and (6.1)). □

## Supplement S.II: Remaining proofs of the results in Section 2.4

**Proof of Equation (2.10).** In order to show that the stochastically increasing rearrangement $C^\uparrow$ is indeed a copula, we verify the properties (1) to (3) of Definition S.I.1.

1. It follows from $(K_C)^*(u, 0) = 0^* = 0$ that $C^\uparrow(u, 0) = 0$. The identity $C^\uparrow(0, v) = 0$ is trivial by Equation (2.9).
2. By definition, we have

$$C^\uparrow(u, 1) = \int_0^u (K_C)^*(s, 1) \, ds = \int_0^u 1^* \, ds = u .$$

In view of Proposition S.I.6(3), we further obtain that

$$C^{\uparrow}(1, v) = \int_0^1 (K_C)^*(s, v) \, ds = \int_{[0,1]^d} (K_C)^*(\sigma_v(s), v) \, d\mu_{C^{1\cdots d}}(s)$$

$$= \int_{[0,1]^d} K_C(s, [0, v]) \, d\mu_{C^{1\cdots d}}(s) = C(1, \ldots, 1, v) = v \, .$$

3. Since $K_C(s, \cdot)$ is a probability measure for each $s$, we see that $0 \le K_C(\cdot, [0, v_1]) \le K_C(\cdot, [0, v_2])$ whenever $v_1 \le v_2$. Combining this with Proposition S.I.6(2) yields $(K_C)^*(\cdot, v_1) \le (K_C)^*(\cdot, v_2)$. Thus, the $C^{\uparrow}$-volume of a rectangle $[u_1, u_2) \times [v_1, v_2)$ satisfies

$$V_{C^{\uparrow}}([u_1, u_2) \times [v_1, v_2)) = C^{\uparrow}(u_2, v_2) - C^{\uparrow}(u_1, v_2) - C^{\uparrow}(u_2, v_1) + C^{\uparrow}(u_1, v_1)$$

$$= \int_{u_1}^{u_2} (K_C)^*(s, v_2) - (K_C)^*(s, v_1) \, ds \ge 0.$$

Thus $C^{\uparrow}$ is a bivariate copula. Moreover, since $s \mapsto (K_C)^*(s, v)$ is decreasing, $C^{\uparrow}$ is stochastically increasing.

**Proof of (M 1.1)–(M 1.3).** We will now prove properties (M 1.1)–(M 1.3) for $R_\mu(C) = \mu(C^{\uparrow})$. For this, we say that the copula $C$ is *completely dependent* if there exists a measurable function $f$ such that $V = f(U)$. It is proven in Griessenberger, Junker and Trutschnig (2022) that $C$ is completely dependent if, and only if, there is a $\mu_{C^{1\cdots d}}$-$\lambda$-preserving transformation $h : [0, 1]^d \to [0, 1]$ such that

$$K_C(s, F) = \mathbb{1}_F(h(s)) \tag{1}$$

for all Borel measurable sets $F$.

(M1.1) Since $\mu$ only takes values between 0 and 1, we obtain the first assertion.
(M1.2) Suppose $U$ and $V$ are independent. Then, following the proof of Theorem 5.6 in Griessenberger, Junker and Trutschnig (2022), $K_C(s, [0, v]) = v$. Thus, $(K_C)^*(s, v) = v$, $C^{\uparrow} = \Pi$ and $\mu(C^{\uparrow}) = \mu(\Pi) = 0$. If, on the other hand, $\mu(C^{\uparrow}) = 0$, we conclude $C^{\uparrow} = \Pi$ by the properties of $\mu$. But then, $(K_C)^*(s, v) = v$ and $K_C(s, [0, v]) = (K_C)^*(\sigma_v(s), v) = v$ for some measure-preserving transformation $\sigma_v$. Therefore, $C(u, v) = vC^{1\cdots d}(u)$ and $U$ and $V$ are independent.
(M1.3) If $C$ is completely dependent, then $C^{\uparrow} = C^+$ and $\mu(C^{\uparrow}) = \mu(C^+) = 1$ by definition. On the other hand, $\mu(C^{\uparrow}) = 1$ implies $C^{\uparrow} = C^+$ by the properties of $\mu$. Thus,

$$K_C(s, [0, v]) = \partial_1 C^+(\sigma_v(s), v) = \mathbb{1}_{[0,v]}(\sigma_v(s)) \, ,$$

so $C$ is completely dependent by (1).

**Proof of the information gain inequality.** To shorten the necessary notation, we consider the case of three random variables $X_1, X_2$ and $Y$. The copulas of $(X_1, Y)$, $(X_1, X_2, Y)$ and $(X_1, X_2)$ are denoted by $D$, $C$ and $C^{12}$, respectively. Following the proof of Theorem 5.6(5) in Griessenberger, Junker and Trutschnig (2022), we have by the disintegration theorem

$$K_D(u_1, [0, v]) = \int_0^1 K_C(u_1, u_2, [0, v]) \, K_{C^{12}}(u_1, du_2) \, ,$$

where $K_C$ denotes the disintegration with respect to the first two variables as above. Since $K_{C^{12}}(u_1, \cdot)$ is a probability measure for all $u_1 \in [0, 1]$, we have via Jensen's inequality for any convex function $\phi$

$$\int_0^1 \phi(K_D(u_1, [0, v])) \, du_1 = \int_0^1 \phi\left(\int_0^1 K_C(u_1, u_2, [0, v]) \, K_{C^{12}}(u_1, du_2)\right) du_1$$

$$\leq \int_0^1 \int_0^1 \phi\left(K_C(u_1, u_2, [0, v])\right) K_{C^{12}}(u_1, du_2) \, du_1$$

$$= \int_{[0,1]^2} \phi\left(K_C(u_1, u_2, [0, v])\right) \, d\mu_{C^{12}}(u_1, u_2) \, .$$

By Theorem 2.5 of Chong (1974), we have $K_D(\cdot, [0, v]) \leq K_C(\cdot, [0, v])$ for all $v \in [0, 1]$ and therefore $D^\uparrow \leq C^\uparrow$.

**Proof of the conditional independence property.** Suppose $X_2, \ldots, X_d$ and $Y$ are conditionally independent given $X_1$. Using the proof of Proposition 5.9 in Griessenberger, Junker and Trutschnig (2022), we have

$$K_C(u_1, \ldots, u_d, [0, v]) = K_D(u_1, [0, v]) \, ,$$

where $D$ denotes the copula of $X_1$ and $Y$ as above. Thus, similarly to the proof of the information gain inequality, we have for any convex function $\phi$

$$\int_{[0,1]} \phi\left(K_D(u_1, [0, v])\right) \, du_1 = \int_{[0,1]^d} \phi\left(K_D(u_1, [0, v])\right) \, d\mu_{C^{1\cdots d}}(u_1, \ldots, u_d)$$

$$= \int_{[0,1]^d} \phi\left(K_C(u_1, \ldots, u_d, [0, v])\right) \, d\mu_{C^{1\cdots d}}(u_1, \ldots, u_d) \, .$$

and therefore $K_D^*(\cdot, [0, v]) = K_C^*(\cdot, [0, v])$ and $D^\uparrow = C^\uparrow$.



\end{document}